\documentclass[11pt]{amsart}
\usepackage{amsmath,amssymb,amsfonts}
\usepackage{bbm}
\usepackage{enumerate}
\usepackage{amscd, latexsym, leftindex}
\usepackage[all]{xy}
\usepackage{pb-diagram}
\usepackage{mathrsfs}
\usepackage{pictexwd,dcpic}
\usepackage{hyperref,url}
\usepackage{color}

\usepackage{pb-diagram}
\theoremstyle{plain}

\newcounter{theorem}[section]
\numberwithin{theorem}{section}
\numberwithin{equation}{subsection}

\newtheorem{Thm}[theorem]{Theorem}
\newtheorem{Conj}[theorem]{Conjecture}
\newtheorem{Cor}[theorem]{Corollary}
\newtheorem{Prop}[theorem]{Proposition}
\newtheorem{Lem}[theorem]{Lemma}

\theoremstyle{remark}
\newtheorem{Def}[theorem]{Definition}
\newtheorem{Rmk}[theorem]{Remark}

\newcommand{\Hom}{\operatorname{Hom}}

\newcommand{\GL}{\operatorname{GL}}

\newcommand{\SL}{\operatorname{SL}}

\newcommand{\Sym}{\operatorname{Sym}}
\newcommand{\Ind}{\operatorname{Ind}}
\newcommand{\ind}{\operatorname{ind}}

\newcommand{\Para}{\operatorname{Para}} 
\newcommand{\HC}{\operatorname{HC}}

\newcommand{\C}{\mathbb C}

\newcommand{\Z}{\mathbb{Z}}
\newcommand{\R}{\mathbb{R}}

\newcommand{\bm}{\begin{multline*}}
\newcommand{\tu}{\end  {multline*}}

\newcommand{\Smod}{\mathcal{S}\mathrm{mod}}
\newcommand{\SH}{{\rm H}}
\newcommand{\Speh}{{\rm Speh}}
\newcommand{\Gr}{{\rm Gr}}
\newcommand{\std}{{\rm std}}

\def\calE{\mathcal{E}}

\def\calN{\mathcal{N}}
\def\calO{\mathcal{O}}

\def\calS{\mathcal{S}}

\def\calV{\mathcal{V}}

\def\calX{\mathcal{X}}

\def\calZ{\mathcal{Z}}

\def\scrL{\mathscr{L}}

\def\gothg{\mathfrak{g}}

\def\lra{\longrightarrow}

\title[Proof of non-tempered GGP: Archimedean general linear case]
{Non-tempered Gan--Gross--Prasad conjecture for Archimedean general linear groups}
\author{Cheng Chen and Rui Chen} 

\address{Institut de Math\'ematiques de Jussieu--Paris Rive Gauche, B\^atiment Sophie Germain, 8 place Aur\'elie Nemours, 75013 Paris}
\email{cheng.chen@imj-prg.fr}

\address{School of Mathematical Sciences, Zhejiang University, Zijingang Campus, 866 Yuhangtang Road, Hangzhou 310058, China}
\email{rchenmat@zju.edu.cn}

\begin{document}

\begin{abstract}
The local non-tempered Gan--Gross--Prasad conjecture 
suggests that, for a pair of irreducible Arthur type representations of two successive general linear groups, they have non-zero Rankin--Selberg period if and only if they are ``relevant''. In this paper, we prove the ``period implies relevance'' direction of this conjecture for general linear groups over Archimedean local fields. 
Our proof is based on a previous result of Gourevitch--Sayag--Karshon on the annihilator varieties of distinguished representations. 
Combining with a recent result of P. Boisseau on the direction ``relevance implies period'', 
this conjecture for general linear groups over Archimedean local fields is settled.
\end{abstract}

\maketitle

\section{Problem and Set up} 
\label{sec:problem_and_set_up}

A much-studied branching problem in the representation theory of reductive groups over a local field $F$ of characteristic zero is the so-called (corank one) Rankin--Selberg model, also known as the general linear Gan--Gross--Prasad model, which concerns two successive general linear groups. Let $G_n=\GL_n\left(F\right)$ be the general linear group of rank $n$. Regard $G_n$ as a subgroup of $G_{n+1}$ via the map
\[
    g\longmapsto \left(\begin{array}{cc}
                            g & ~\\
                            ~ & 1
                        \end{array}\right)
\]
for $g\in G_n$. One would like to determine the space
\[
    \Hom_{G_n}\left(\pi,\sigma\right)
\]
for irreducible representations $\pi$ of $G_{n+1}$ and $\sigma$ of $G_n$. It is well-known that this Hom space has dimension at most one \cite{MR2529937} \cite{MR2680495} \cite{MR2874638}. Moreover, by the work of Jacquet--Piatetskii-Shapiro--Shalika \cite{MR701565} and Jacquet \cite{MR2533003}, one knows that when $\pi$ and $\sigma$ are both generic, then the Hom space is always non-zero.

\vskip 5pt

To understand the Hom space when $\pi$ and $\sigma$ are not necessarily generic, recently Gan, Gross and Prasad proposed a conjecture \cite[Conj. 5.1]{MR4190046} for representations of \textit{Arthur type}. They introduced a mysterious notion called ``\textit{relevance}'' of two A-parameters, and conjectured that when $\pi$ and $\sigma$ are of Arthur type, the Hom space is non-zero if and only if their A-parameters are relevant. Soon after, the non-Archimedean case of this conjecture has been settled: M. Gurevich proved one direction in \cite{MR4375451}, and Chan proved the conjecture in full generality \cite{MR4373242}. Both these works use the Bernstein--Zelevinsky derivative as an important tool. Later, Chan also extended his result to all irreducible representations, see \cite{chan2022quotient} for more details.

\vskip 5pt

However, for the Archimedean case, the conjecture is still open. One reason is that, in the Archimedean case, the theory of Bernstein--Zelevinsky derivative is much more complicated. In \cite[Thm. E]{MR4291954}, Gourevitch, Sayag, and Karshon obtained some partial results from a different angle. By investigating the annihilator variety, they provided a ``\textit{micro-local}'' condition for an irreducible Casselman--Wallach representation to be distinguished. Since annihilator varieties of irreducible unitary representations determine their Arthur $\SL_2$-types \cite[Thm. B]{MR3029948}, the conclusion of \cite[Thm. E]{MR4291954} can be translated into the language of A-parameters; hence, it provided affirmative evidence for the Archimedean case of the non-tempered Gan--Gross--Prasad conjecture.

\vskip 5pt

The main purpose of this short paper is to upgrade the result of Gourevitch--Sayag--Karshon \cite[Thm. E]{MR4291954} to a proof for the non-tempered Gan--Gross--Prasad conjecture \cite[Conj. 5.1]{MR4190046}, \textit{without} using the theory of derivatives. Our main tool is the Mackey theory (or, more precisely, Borel's lemma). To describe our results, we need to make some conventions. From now on, let $F$ be an Archimedean local field, i.e. $F=\R$ or $\C$. We shall work in the setting of \textit{almost linear Nash groups} as in \cite{MR4211018}. By ``representations'' of a Nash group,  we mean ``smooth Fr\'echet representations of moderate growth.'' In this paper we shall mainly deal with \textit{unitary} representations. As usual, here the terminology ``unitary representation'' really means ``unitarizable smooth Fr\'echet representations of moderate growth.'' The main result of this paper is the ``period implies relevance'' direction of \cite[Conj. 5.1]{MR4190046}.

\begin{Thm}\label{T:MainNTGGPGL}
Let $\pi$ and $\sigma$ be irreducible unitary representations of $G_{n+1}$ and $G_n$, respectively.  
If the Hom space 
    \[
        \Hom_{G_n}\left(\pi,\sigma\right) \neq 0,
    \]
    then $\pi$ and $\sigma$ are relevant as in Definition \ref{D:unitaryRelevant}.
\end{Thm}


\vskip 5pt

Here as stated in the theorem, we extend the notion of ``relevance'' to all irreducible unitary representations, not only those of Arthur type.
We should mention that if both $\pi$ and $\sigma$ are of Arthur type, then the notion of ``relevance'' in Definition \ref{D:unitaryRelevant} is the same as the one defined in \cite[Sect. 3]{MR4190046}. To prove the other direction of \cite[Conj. 5.1]{MR4190046}, namely, ``relevance implies period'', one can employ the integral method. This has been achieved in the thesis of P. Boisseau \cite{thesisboisseau} recently. Therefore, combining our result with P. Boisseau's result, we know that the non-tempered Gan--Gross--Prasad conjecture holds for Archimedean general linear groups.

\begin{Cor}
Let $\pi$ and $\sigma$ be irreducible Arthur type representations of $G_{n+1}$ and $G_n$ respectively.  
Then the Hom space 
    \[
        \Hom_{G_n}\left(\pi,\sigma\right)
    \]
is non-zero if and only if $\pi$ and $\sigma$ are relevant. 
\end{Cor}

\vskip 5pt

Although not specifically noted in \cite{MR4190046}, one can also formulate and prove analogous conjectures for Fourier--Jacobi models and higher corank Rankin--Selberg models by using the connection of these models to the corank one Rankin--Selberg model (see Lemma \ref{L:FJToBessel} and Proposition \ref{P:ReductionToCorank1} below).

\vskip 5pt


The paper is organized as follows.
We first recall the basic setup of Schwartz analysis and Borel's lemma in Section \ref{sec:schwartz_homology}, and prove two useful lemmas (Lemma \ref{L:inftesimalArguement} and Lemma \ref{L:vanishingLem1}). These lemmas allow us to conclude certain Schwartz homologies and Hom spaces are vanishing.
In Section \ref{sec:tadic_vogan_classification}, based on Vogan's classification of the unitary dual of Archimedean general linear groups, we extend the definition of relevance to irreducible unitary representations (Definition \ref{D:unitaryRelevant}). After that, a combinatorial description of relevance (Lemma \ref{L:RelevantCombinatorial}) will be given.
In Section \ref{sec:orbit_analysis_and_higher_corank_models}, we investigate how a parabolic induced representation decomposes when restricted to the mirabolic subgroup, by doing some orbit analysis. As applications, we relate Fourier--Jacobi models and higher corank Rankin--Selberg models to the corank one Rankin--Selberg model in Lemma \ref{L:FJToBessel} and Proposition \ref{P:ReductionToCorank1} respectively. All these models will appear in the proof of Theorem \ref{T:MainNTGGPGL}. 
Finally, combining these results with the key input \cite[Thm. E]{MR4291954}, we prove Theorem \ref{T:MainNTGGPGL} in Section \ref{sec:proof_of_the_main_result_i_period_implies_relevance}.

\vskip 5pt

\subsection*{Notations} 
\label{sub:notations}

We denote by $L_F$ the Weil group associated to $F$, and by $|\cdot|$ the normalized absolute value of $F$. For each positive integer $d$, let $S_d$ be the unique irreducible $d$-dimensional (algebraic) representation of $\SL_2\left(\C\right)$. We use ${\rm Mat}_{a,b}\left(F\right)$ to denote the space $a\times b$-matrices over $F$, and by $I_a$ the $a\times a$ identity matrix. The determinant of $G_n$ will be denoted by $\det_n$. When $n$ is clear from the context we may suppress it from the subscript. Let $M_{n}$ be the usual mirabolic subgroup of $G_n$ consisting of elements of the form
\[
    \left(\begin{array}{cc}
                            g & v\\
                            ~ & 1
                    \end{array}\right)
\]
for $g\in G_{n-1}$ and $v\in F^{n-1}$. We also denote by $\omega_n$ the (normalized) Weil representation of $G_n$, i.e. the space of Schwartz functions $\calS\left(F^n\right)$ equipped with the action 
    \[
        \left(g\cdot\varphi\right) \left(v\right) = |\det g|^{-\frac{1}{2}}\cdot\varphi\left(g^{-1}v\right)
    \]
for $g\in G_n$ and $v\in F^n$. For two representations $\pi$ of $G_n$ and $\sigma$ of $G_m$, we use the standard notation $\pi\times\sigma$ to denote the normalized parabolic induction of $\pi\boxtimes\sigma$ to $G_{n+m}$ with respect to the standard parabolic subgroup whose Levi component is $G_n\times G_m$. All tensors in this paper should be understood as completed tensors.


\vskip 5pt

\subsection*{Acknowledgements}

This project was initiated during the scientific programme `Representation Theory and Noncommutative Geometry', supported by the Institut Henri Poincaré (UAR 839 CNRS-Sorbonne Université), and LabEx CARMIN (ANR-10-LABX-59-01). We would like to thank Paul Boisseau for helpful discussions and explaining to us his thesis work. We also thank Kei Yuen Chan and Maxim Gurevich for several email correspondences.
The first author is supported by the European Union’s Horizon 2020 research and innovation programme under the Marie Skłodowska-Curie grant agreement No. 101034255. 

\vskip 10pt

\section{Preparations for the proof} 
\label{sec:schwartz_homology}

\subsection{Schwartz analysis and Borel's lemma} 
\label{sub:toolkits}

As mentioned in the previous section, we work in the setting of almost linear Nash groups.  In the sense of \cite{MR4211018}, an almost linear Nash group is a group with a Nash structure, which admits a homomorphism to some general linear group with a finite kernel. All groups we are working with in this paper (general linear groups and their closed subgroups) belong to this class of groups.

\vskip 5pt

Let $G$ be an almost linear Nash group, we denote by $\Smod\left(G\right)$ the category of smooth Fr\'echet representations of moderate growth of $G$. For a representation $\left(\pi, V\right)$ of $G$, we put
\[
    V_G = V\big/\left\langle \pi\left(g\right)v-v~\big|~g\in G,~v\in V\right\rangle,
\]
equipped with the quotient topology. It might not be Hausdorff. The functor $V\longmapsto V_G$ is a right exact functor from $\Smod\left(G\right)$ to the category of topological vector spaces. The Schwartz homology $\SH_i\left(G,V\right)$ is by definition the $i$-th derived functor of it.

\vskip 5pt

Let $H$ be a Nash subgroup of $G$, and $\left(\pi_H, V\right)$ be a representation of $H$. Then there is a continuous map
\[
    I: \mathcal{S}\left(G,V\right) \longrightarrow C^\infty\left(G,V\right), \quad  f\longmapsto \left(g\mapsto \int_H \pi_H(h)f\left(h^{-1}g\right)\,dh\right)
\]
from the space of Schwartz functions $\mathcal{S}\left(G,V\right)$ to the space of smooth functions $C^\infty\left(G,V\right)$, with image landing in the usual un-normalized induction $\leftindex^u{\Ind}_H^G V$. The image of $I$, equipped with the quotient topology of $\mathcal{S}\left(G,V\right)$, is called the Schwartz induction of $\pi_H$ and will be denoted by $\leftindex^u\ind_H^G V$.

\begin{Rmk}
The following two facts are worth noting:

\begin{itemize}
\item if $H$ is a parabolic subgroup, then, as $H\backslash G$ is compact, the Schwartz induction $\leftindex^u\ind_H^G V$ is the usual un-normalized parabolic induction;
\vskip 5pt

\item the Schwartz induction $\leftindex^u\ind_H^G V$ can be also intepreted as the space of Schwartz sections over the vector bundle
\[
    H\backslash (G\times V) \longrightarrow H\backslash G,
\]
where $H$ acts on $G\times V$ diagonally.
\end{itemize}
\end{Rmk}

\vskip 5pt

The next proposition is a version of Frobenius reciprocity in the setting of Schwartz analysis. Readers may consult \cite[Prop. 6.6, Thm. 6.8]{MR4211018} for more details.

\begin{Prop}\label{P:Frobenious}
Let $V$ be a representation of $H$. Then
\[
   \left(\left(\leftindex^u\ind_H^G \left( V\cdot \delta_H\right)\right)\cdot\delta_G^{-1}\right)_G \simeq V_H,
\]
where $\delta_H$ and $\delta_G$ are the modulus characters of $H$ and $G$, respectively.
\end{Prop}

\vskip 5pt

In later proofs we also need to do some orbit analysis. We now introduce a version of Borel's lemma, following \cite[Prop. 8.2, 8.3]{MR4211018} and \cite[Prop. 2.5]{xue2020besselShomology}. Let $\mathcal{X}$ be a Nash manifold and $\mathcal{Z}$ a closed Nash submanifold. We put $\mathcal{U} = \mathcal{X} \backslash \mathcal{Z}$. Let $\mathcal{E}$ be a tempered vector bundle over $\mathcal{X}$, then the ``extension by zero'' of Schwartz sections over $\mathcal{U}$ gives us the following exact sequence:
\[
    0 \longrightarrow \Gamma\left(\mathcal{U}, \mathcal{E}\right) \longrightarrow \Gamma\left(\mathcal{X}, \mathcal{E}\right) \longrightarrow \Gamma_{\mathcal{Z}}\left(\mathcal{X}, \mathcal{E}\right) \longrightarrow 0,
\]
where the symbol $\Gamma$ stands for taking Schwartz sections, and
\[
    \Gamma_{\mathcal{Z}}\left(\mathcal{X}, \mathcal{E}\right) := \Gamma\left(\mathcal{X}, \mathcal{E}\right)\big/\Gamma\left(\mathcal{U}, \mathcal{E}\right).
\]
Let $N_{\mathcal{Z}|\mathcal{X}}^\vee$ 
 be the conormal bundle of $\mathcal{Z}$ and $\mathcal{N}_{\mathcal{Z}|\mathcal{X}}^\vee$ the complexification of $N_{\mathcal{Z}|\mathcal{X}}^\vee$.

\begin{Prop}\label{P:BorelLemma}
There is a descending filtration on $\Gamma_{\mathcal{Z}}\left(\mathcal{X}, \mathcal{E}\right)$ indexed by $k\in\mathbb{N}:$
\[
    \Gamma_{\mathcal{Z}}\left(\mathcal{X}, \mathcal{E}\right)=\Gamma_{\mathcal{Z}}\left(\mathcal{X}, \mathcal{E}\right)_0 \supset \Gamma_{\mathcal{Z}}\left(\mathcal{X}, \mathcal{E}\right)_1 \supset \Gamma_{\mathcal{Z}}\left(\mathcal{X}, \mathcal{E}\right)_2 \supset \cdots,
\]
such that
\[
    \Gamma_{\mathcal{Z}}\left(\mathcal{X}, \mathcal{E}\right) = \lim_{\longleftarrow} \Gamma_{\mathcal{Z}}\left(\mathcal{X}, \mathcal{E}\right)\big/\Gamma_{\mathcal{Z}}\left(\mathcal{X}, \mathcal{E}\right)_k,
\]
with each graded piece isomorphic to
\[
    \Gamma\left(\mathcal{Z}, \Sym^k\mathcal{N}_{\mathcal{Z}|\mathcal{X}}^\vee\otimes\mathcal{E}\big|_\mathcal{Z}\right), \quad k\in\mathbb{N}.
\]
Moreover, if $\mathcal{X}$ is a $G$-Nash manifold, $\mathcal{Z}$ is stable under the $G$-action, and $\mathcal{E}$ is a tempered $G$-bundle, then this filtration is stable under the $G$-action.

\end{Prop}

\vskip 5pt

\subsection{Infinitesimal character argument and a vanishing lemma} 
\label{sub:two_vanishing_lemmas}


Finally, we introduce two lemmas, which will help us to conclude certain Hom spaces or Schwartz homologies are zero. Recall that for any irreducible representation $\pi$ of $G_n$, there exists a parabolic induction
\[
    \tau_1|{\det}_{d_1}|^{s_1}\times\tau_2|{\det}_{d_2}|^{s_2}\times\cdots\times\tau_r|{\det}_{d_r}|^{s_r},
\]
where $\tau_i$ are unitary discrete series of $G_{d_i}$, and 
\[
    s_1\geq s_2\geq\cdots\geq s_r
\] 
are real numbers, such that $\pi$ is the unique quotient of this parabolic induction, called the ``\textit{Langlands quotient}''. We denote it by $\pi = LQ\left(\tau_1|{\det}_{d_1}|^{s_1}\times\tau_2|{\det}_{d_2}|^{s_2}\times\cdots\times\tau_r|{\det}_{d_r}|^{s_r}\right)$. It is known that $\pi$ occurs in this parabolic induction with multiplicity one. We first prove a result for Schwartz homologies of certain parabolic induced representations, which is based on a standard infinitesimal character argument.

\begin{Lem}\label{L:inftesimalArguement}
Let $\pi$
be a finite length representation of $G_n$. Let $\chi_1,\chi_2,\cdots,\chi_k$ be characters of $G_1$, $F$ a finite dimensional representation of $G_k$, and $\pi_0$ a representation of $G_{n-d}$ (not necessarily of finite length). Then there exists a Zariski closed proper subset ${\rm Ex}$ of $\C^n$, such that for all $\left(s_1, s_2, \cdots, s_{k}\right) \in \C^n\,\backslash\, {\rm Ex}$, we have
\[
    \SH_i\left(G_n, \pi^\vee\otimes \left(I_{\underline{\chi}}\left(s_1,s_2,\cdots,s_{k}\right)\otimes F\right)\times\pi_0\right) = 0
\]
for all $i\geq 0$. Here $I_{\underline{\chi}}\left(s_1,s_2,\cdots,s_{k}\right)\coloneqq\chi_1|\cdot|^{s_1}\times\chi_2|\cdot|^{s_2}\times\cdots\times\chi_k|\cdot|^{s_{k}}$.
\end{Lem}

\vskip 5pt

\begin{proof}
This lemma can be proved in the same way as \cite[Lem. 4.2]{xue2020besselShomology}. The idea is to construct an element $z$ in the center of the universal enveloping algebra of $\gothg_n = {\rm Lie}\left(G_n\right)_\C$ explicitly, such that $z$ annihilates $\left(I_{\underline{\chi}}\left(s_1,s_2,\cdots,s_{k}\right)\otimes F\right)\times\pi_0$ but does not annihilate $\pi$.

\end{proof}

\vskip 5pt

Next, we prove a vanishing lemma that concerns Hom spaces.

\begin{Lem}\label{L:vanishingLem1}
Let
\[
    \pi = LQ\left(\tau_1|{\det}_{d_1}|^{s_1}\times\tau_2|{\det}_{d_2}|^{s_2}\times\cdots\times\tau_r|{\det}_{d_r}|^{s_r}\right)
\]
be an irreducible representation of $G_n$. Let $s$ be a real number, $\tau$ a unitary discrete series of $G_d$, and $\pi_0$ a representation of $G_{n-d}$ (not necessarily of finite length). Suppose there either 

\begin{itemize}
    \item $s > s_1$; or

    \vskip 5pt

    \item $s+\frac{1}{2}>s_1\geq\cdots \geq s_k \geq s > s_{k+1}$, and $\tau|{\det}_d|^s\not\simeq \tau_i|{\det}_{d_i}|^{s_i}$ for $i=1,\cdots,k$.
\end{itemize}
\vskip 5pt
Then we have
\[
    \Hom_{G_n}\left(\tau|\det|^s\times\pi_0, \pi\right) = 0.
\]
\end{Lem}

\vskip 5pt

\begin{proof}
This lemma can be proved in the same manner as \cite[Thm. A.1.1]{chen2023multiplicity} following the idea of \cite{moeglin2012conjecture}. For the sake of completeness, we provide a sketch of the proof below.

\vskip 5pt

We shall prove by contradiction. Suppose on the contrary that there is a real number $s$, a unitary discrete series $\tau$ of $G_d$ and a representation $\pi_0$ of $G_{n-d}$, satisfying the conditions in the lemma, but
\[
\Hom_{G_n}(\tau|\det|^s\times\pi_0,\pi) \neq 0.
\]
Let $\gothg_n$ be the complexified Lie algebra of $G_n$, and $K_n$ a maximal compact subgroup of $G_n$. Since for any representation in $\Smod\left(G_n\right)$, its subspace of $K_n$-finite vectors is dense, we know that the Hom space between the $\left(\gothg_n, K_n\right)$-module of $\tau|\det|^s\times\pi_0$ and the $\left(\gothg_n, K_n\right)$-module of $\pi$ is also non-zero.
Then we can appeal to the second adjointness. In \cite[Sect. 5]{MR4344027}, Y. Din defined various functors. In particular, the composition ${\rm cofib}_I\circ\mathcal{C}_I$ of the functors ${\rm cofib}_I$ and $\mathcal{C}_I$ there, provides a right adjoint functor of the parabolic induction functor ${\rm pind}_I$. This functor ${\rm cofib}_I\circ\mathcal{C}_I$ also preserves admissibility and the property of finite length. 
By using this functor, we can further assume that $\pi_0$ is an irreducible representation.

\vskip 5pt

Next we use the Langlands classification to deduce a contradiction. Suppose that
\[
    \pi_0 = LQ\left(\rho_1|\det|^{b_1}\times\rho_2|\det|^{b_2}\times\cdots\times\rho_m|\det|^{b_m}\right),
\]
where $\rho_i$ are unitary discrete series of $G_{\alpha_i}$, and $b_1\geq\cdots\geq b_m$ are real numbers. Substituting this into the Hom space, we have
\[
    \Hom_{G_n}\left(\tau|\det|^s\times\rho_1|\det|^{b_1}\times\rho_2|\det|^{b_2}\times\cdots\times\rho_m|\det|^{b_m},\pi\right) \neq 0.
\]
If $s\geq b_1$, then $\tau|\det|^s\times\rho_1|\det|^{b_1}\times\cdots\times\rho_m|\det|^{b_m}$ is a standard module of $G_n$. It then follows that
\[
    \pi = LQ\left(\tau|\det|^s\times\rho_1|\det|^{b_1}\times\rho_2|\det|^{b_2}\times\cdots\times\rho_m|\det|^{b_m}\right),
\]
which contradicts our assumption on $\pi$. So we must have $s< b_1$. By induction in stages, one can see that there is an irreducible subquotient $T$ of $\tau|\det|^s\times \rho_1|\det|^{b_1}$, such that
\[
    \Hom_{G_n}\left(T\times\rho_2|\det|^{b_2}\times\cdots\times\rho_m|\det|^{b_m},\pi\right) \neq 0.
\]
Again, we write $T$ as a Langlands quotient
\[
    T=LQ\left(\rho'_0|\det|^{b'_0}\times\cdots\times \rho'_t|\det|^{b'_t}\right).
\]
If $\tau|\det|^s\times \rho_1|\det|^{b_1}$ is irreducible, then $T \simeq \rho_1|\det|^{b_1}\times\tau|\det|^s$. If $\tau|\det|^s\times \rho_1|\det|^{b_1}$ is reducible, then by the irreducibility criterion for principal series \cite[Thm. 6.2]{MR401654} \cite[Thm. 10b]{MR1476497}, we can conclude that $b'_0\geq s+1/2$.
Replacing $T$ in the Hom space by its standard module $\rho'_0|\det|^{b'_0}\times\cdots\times \rho'_t|\det|^{b'_t}$ and repeating this procedure. After finitely many steps, we will find that
\[
\Hom_{G_n}\left(\rho''_0|\det|^{c_0}\times\rho''_1|\det|^{c_1}\times\cdots\times\rho''_u|\det|^{c_u},\pi\right) \neq 0,
\]
where $\rho''_0|\det|^{c_0}\times\rho''_1|\det|^{c_1}\times\cdots\times\rho''_u|\det|^{c_u}$ is a standard module, such that either $c_0\geq s+1/2$, or $\tau|\det|^s\simeq \rho''_i|\det|^{c_i}$ for some $1\leq i \leq u$. This again contradicts our assumption on $\pi$. Therefore, the lemma holds.

\end{proof}

\vskip 5pt

\vskip 10pt

\section{Vogan's classification} 
\label{sec:tadic_vogan_classification}

\subsection{Building blocks} 
\label{sub:building_blocks_of_the_unitary_dual}

According to Vogan's classification of the unitary dual of Archimedean general linear groups \cite{MR827363}, any irreducible unitary representation of $G_n$ is of the form
\[
    \pi \simeq \pi_1 \times \pi_2 \times \cdots \times \pi_r,
\]
where each $\pi_i$ is one of the following building blocks:
\vskip 5pt

\begin{itemize}
    \item a generalized Speh representation of $G_{kd}$, such a representation is of the form 
    \[
        \Speh\left(\delta,d\right) \coloneqq LQ\left(\delta|\det|^{\frac{d-1}{2}}\times \delta|\det|^{\frac{d-3}{2}}\times \cdots \times \delta|\det|^{-\frac{d-1}{2}}\right),
    \]
    where $\delta$ is a unitary discrete series of $G_k$;

    \vskip 5pt

    \item a complementary series representation of $G_{2kd}$, such a representation is of the form
    \[
        C\left(\delta,d,s\right) \coloneqq \Speh\left(\delta,d\right)|\det|^{s}\times \Speh\left(\delta,d\right)|\det|^{-s},
    \]
    where $0<s<1/2$ is a real number.
\end{itemize}
\vskip 5pt
An irreducible unitary representation $\pi$ of $G_n$ is said to be of Arthur type if, in its decomposition, no complementary series shows up. In this case, one can attach an A-parameter 
\[
    \Psi_\pi: L_F\times \SL_2\left(\C\right) \lra G_n\left(\C\right)
\]
to $\pi$: suppose that $\pi_i = \Speh\left(\delta_i,d_i\right)$, then we set
\[
    \Psi_\pi = \scrL\left(\delta_1\right)\boxtimes S_{d_1} + \scrL\left(\delta_2\right)\boxtimes S_{d_2} + \cdots + \scrL\left(\delta_r\right)\boxtimes S_{d_r}.  
\]
Here $\scrL\left(\delta_i\right)$ is the L-parameter attached to $\delta_i$ by the local Langlands correspondence of general linear groups.

\vskip 5pt

To extend the notion of relevance to all irreducible unitary representations, we need an alternative description of the complementary series. Let $\Delta\left(\delta,s\right) \coloneqq \delta|\det|^s\times\delta|\det|^{-s}$. Then 
\begin{equation}\label{E:CompleSeriesStdMod}
    \Delta\left(\delta,s\right)|\det|^{\frac{d-1}{2}} \times \Delta\left(\delta,s\right)|\det|^{\frac{d-3}{2}} \times \cdots \times \Delta\left(\delta,s\right)|\det|^{-\frac{d-1}{2}}
\end{equation}
is a standard module, and we claim that
\[
     C\left(\delta,d,s\right) \simeq LQ\left(\Delta\left(\delta,s\right)|\det|^{\frac{d-1}{2}} \times \Delta\left(\delta,s\right)|\det|^{\frac{d-3}{2}} \times \cdots \times \Delta\left(\delta,s\right)|\det|^{-\frac{d-1}{2}}\right).
\]
To see this, note that by definition $C\left(\delta,d,s\right)$ is a quotient of 
\begin{equation}\label{E:CompleSeriesQuoDef}
    \left(\delta|\det|^{s+\frac{d-1}{2}}\times \cdots \times \delta|\det|^{s-\frac{d-1}{2}}\right)\times \left(\delta|\det|^{-s+\frac{d-1}{2}}\times \cdots \times \delta|\det|^{-s-\frac{d-1}{2}}\right).
\end{equation}
By a result of Casselman--Shahidi (see \cite[Prop. 5.3]{MR1634020}), for every pair of non-negative integers $\left(i,j\right)$, we have
\[
    \delta|\det|^{s+\frac{d-1}{2}-i}\times\delta|\det|^{-s+\frac{d-1}{2}-j} \simeq \delta|\det|^{-s+\frac{d-1}{2}-j} \times \delta|\det|^{s+\frac{d-1}{2}-i}
\]
is irreducible. Therefore, one can exchange the order of the induction in (\ref{E:CompleSeriesQuoDef}) one at a time, and eventually get (\ref{E:CompleSeriesStdMod}) from it. A worth noting corollary is that there is a surjection
\[
    \delta|\det|^{s+\frac{d-1}{2}} \times C\left(\delta,d-1,\frac{1}{2}-s\right) \times \delta|\det|^{-s-\frac{d-1}{2}} \twoheadrightarrow C\left(\delta,d,s\right). 
\]

\vskip 5pt


\subsection{Relevance condition for unitary representations} 
\label{sub:formal_parameters_and_relevance}

Now, for any pair $\left(\delta,s\right)$, where $\delta$ is a discrete series representation of a general linear group, and $s\in\R$ such that $0<s<\frac{1}{2}$, we attach a symbol $\scrL\left(\delta,s\right)$ to it. Similar to the case of Speh representations, for a complementary series $C\left(\delta,d,s\right)$ of $G_{2kd}$, we attach a formal parameter $\scrL\left(\delta,s\right)\boxtimes S_d$ to it. 
We set 
\[
    D\left(\scrL\left(\delta,s\right)\right) \coloneqq \scrL\left(\delta,\frac{1}{2}-s\right).
\]
We denote by $\Para_{disc}$ the set of all L-parameters $\scrL\left(\delta\right)$ of discrete series, and by $\Para_{cmpl}$ the set of all symbols $\scrL\left(\delta,s\right)$. Set $\Para = \Para_{disc} \sqcup \Para_{cmpl}$, and $\Psi_{unit}\left(\GL\right)$ the set of all finite formal sums of the form
\[
    a_1 \scrL_1\boxtimes S_{d_1} + a_2 \scrL_2 \boxtimes S_{d_2} + \cdots + a_r \scrL_r \boxtimes S_{d_r},
\]
where each $a_i\in\Z_{>0}$, $\scrL_i\in \Para$ and $d_i\in \Z_{>0}$. Then $\Psi_{unit}\left(\GL\right)$ is a monoid with obvious addition and zero.
For an irreducible unitary representation
\[
    \pi = \pi_1 \times \pi_2 \times \cdots \times \pi_r
\]
of $G_n$, we attach a formal sum in $\Psi_{unit}\left(\GL\right)$
\[
    \Psi_\pi = \Psi_1 + \Psi_2 + \cdots + \Psi_r
\]
to it, where $\Psi_i = \scrL\left(\delta\right)\boxtimes S_d$ is the A-parameter associated to $\pi_i$ if $\pi_i = \Speh\left(\delta,d\right)$ is a Speh representation, and $\Psi_i = \scrL\left(\delta,s\right)\boxtimes S_{d}$ is the parameter defined as above if $\pi_i = C\left(\delta,d,s\right)$ is a complementary series. It is easy to verify that the parameter $\Psi_\pi$ is well-defined and, conversely, uniquely determines $\pi$. 
If $\pi$ is generic, we say $\Psi_\pi$ is a generic parameter. When $\pi$ is of Arthur type, the formal sum we defined recovers the A-parameter of $\pi$. 

\vskip 5pt

As in \cite[Def. 1]{MR2133760}, we can associate a partition of $n$ to $\pi$, called the \textit{$\SL_2$-type} of $\pi$ and denoted by $TP\left(\pi\right)$, as follows: 
\vskip 5pt
\begin{itemize}
     \item $TP\left(\pi\right)$ is obtained by concatenating $TP\left(\pi_i\right)$ of all $\pi_i$ occurring in its decomposition;

     \vskip 5pt

     \item if $\pi_i=\Speh\left(\delta_i,d_i\right)$, and $\delta_i$ is a unitary discrete series representation of $G_{k_i}$, then the $\SL_2$-type $TP\left(\pi_i\right)$ of $\pi_i$ is 
     \[
        \big(\underbrace{d_i,d_i,\cdots,d_i}_{k_i\textit{-times}}\big);
     \]

     \vskip 5pt

     \item if $\pi_i=C\left(\delta_i,d_i,s_i\right)$, and $\delta_i$ is a unitary discrete series representation of $G_{k_i}$, then the $\SL_2$-type $TP\left(\pi_i\right)$ of $\pi_i$ is 
     \[
        \big(\underbrace{d_i,d_i,\cdots,d_i}_{2k_i\textit{-times}}\big).
     \]

     \vskip 5pt
 \end{itemize} 
In particular, an irreducible unitary representation $\pi$ is generic if and only if its $\SL_2$-type is of the form
\[
    \left(1,1,\cdots,1\right).
\]

\begin{Def}\label{D:unitaryRelevant}
Let $\pi$ be an irreducible unitary representation of $G_n$, and $\sigma$ an irreducible unitary representation of $G_m$. Suppose that
\begin{equation}\label{E:DecompPara}
    \Psi_\pi = \scrL_1\boxtimes S_{d_1} + \scrL_2\boxtimes S_{d_2} + \cdots + \scrL_r\boxtimes S_{d_r},
\end{equation}  
where each $\scrL_i\in\Para$, namely, $\scrL_i$ is either the L-parameter $\scrL\left(\delta_i\right)$ of a discrete series $\delta_i$, or the symbol $\scrL\left(\delta_i,s_i\right)$ corresponding to a complementary series $C\left(\delta_i,s_i,d_i\right)$. 
\begin{enumerate}
    \item We set
    \[
        NT_{\pi} \coloneqq \sum_{i=1}^{r} \left(d_i-1\right),
    \]
which measures the ``non-temperedness'' of $\pi$.
    \vskip 5pt

    \item We say that $\pi$ and $\sigma$ are relevant, if there is a partition $\left\{1,2,\cdots, r\right\} = I \sqcup J \sqcup K$, such that for any $k\in K$, $\scrL_k\in\Para_{cmpl}$, and
\begin{equation}\label{E:RelevantDecompPara}
    \Psi_\sigma = \sum_{i\in I}\scrL_i\boxtimes S_{d_i+1} +\sum_{j\in J}\scrL_j\boxtimes S_{d_j-1} + \sum_{k\in K} D\left(\scrL_k\right) \boxtimes S_{d_k} + \Psi_0.
\end{equation}
Here $\Psi_0$ is a generic parameter, and summands of the form $\scrL\boxtimes S_0$ are understood as zero. 
\end{enumerate}

\end{Def}

\vskip 5pt

 When $\pi$ and $\sigma$ are both of Arthur type, then our definition of relevance recovers Gan--Gross--Prasad's definition in \cite[Sect. 3]{MR4190046}. 

\vskip 5pt

\begin{Rmk}
The following properties of this notion of relevance are worth noting.
\begin{enumerate}
    \item This notion of relevance is symmetric.

    \vskip 5pt

    \item If both $\pi$ and $\sigma$ are generic, then they are relevant.
\end{enumerate}   
\end{Rmk}

\vskip 5pt

We end up this section by giving a combinatorial description of the notion of relevance. Suppose we are given two irreducible unitary representations $\pi$ and $\sigma$ of general linear groups. We write the parameter of $\pi$ as in (\ref{E:DecompPara}). Let $\eta\in\Para$, namely, $\eta$ is either:

\begin{itemize}
    \item the L-parameter $\scrL\left(\delta\right)$ of a discrete series $\delta$ of $G_k$; or

    \vskip 5pt

    \item the symbol $\scrL\left(\delta,s\right)$ attached to a discrete series $\delta$ and a real number $0<s<1/2$.
\end{itemize}
\vskip 5pt
For any positive integer $a\in\Z_{>0}$ and a subset $S\subset\left\{1,2,\cdots,r\right\}$, we set
\[
    m_S\left(\eta,a;\pi\right) = \sharp\left\{i\in S~\big|~\textit{$\scrL_i = \eta$ and $d_i=a$} \right\}.
\]
When $S=\left\{1,2,\cdots,r\right\}$, we suppress it from the subscript and simply denote this multiplicity by $m\left(\eta,a;\pi\right)$. Likewise, we define $m_S\left(\eta,a;\sigma\right)$ and $m\left(\eta,a;\sigma\right)$. We define two differences $\Lambda\left(\eta,a;\pi,\sigma\right)$ and $\Lambda\left(\eta,a;\sigma,\pi\right)$ as follows.

\vskip 5pt

$\bullet$ If $\eta\in\Para_{disc}$, i.e. $\eta = \scrL\left(\delta\right)$ is the L-parameter of a discrete series, put 
\[
    \Lambda\left(\eta,a;\pi,\sigma\right) = m\left(\eta,a;\pi\right) - m\left(\eta,a+1;\sigma\right) + m\left(\eta,a+2;\pi\right) - m\left(\eta,a+3;\sigma\right) + \cdots.
\]
We define $\Lambda\left(\eta,a;\sigma,\pi\right)$ similarly.

\vskip 5pt

$\bullet$ If $\eta\in\Para_{cmpl}$, i.e. $\eta = \scrL\left(\delta,s\right)$ is the symbol attached to a complementary series, put
\begin{align*}
    \Lambda\left(\eta,a;\pi,\sigma\right) = \Big(m\left(\eta,a;\pi\right) &+ m\left(D\left(\eta\right),a+1;\pi\right) + m\left(\eta,a+2;\pi\right) + \cdots \Big) \\
    &- \Big(m\left(\eta,a+1;\sigma\right) + m\left(D\left(\eta\right),a+2;\sigma\right) + m\left(\eta,a+3;\sigma\right) + \cdots \Big).
\end{align*}
We define $\Lambda\left(\eta,a;\sigma,\pi\right)$ similarly.

\vskip 5pt

\begin{Lem}\label{L:RelevantCombinatorial}
The representations $\pi$ and $\sigma$ are relevant if and only if these alternating sums $\Lambda\left(\eta,a;\pi,\sigma\right)$ and $\Lambda\left(\eta,a;\sigma,\pi\right)$ are non-negative for any $\eta\in\Para$ and $a\in\Z_{>0}$.
\end{Lem}

\vskip 5pt

\begin{proof}
Firstly, we prove that if $\pi$ and $\sigma$ are relevant, then these alternating sums $\Lambda\left(\eta,a;\pi,\sigma\right)$ and $\Lambda\left(\eta,a;\sigma,\pi\right)$ are non-negative. Suppose that the parameter of $\pi$ has a decomposition as in (\ref{E:DecompPara}), and $\left\{1,2,\cdots, r\right\} = I \sqcup J\sqcup K$ is a partition such that (\ref{E:RelevantDecompPara}) holds. 

\vskip 5pt

$\bullet$ If $\eta\in\Para_{disc}$, i.e. $\eta = \scrL\left(\delta\right)$ is the L-parameter of a discrete series, for any non-negative integer $k$ by the decomposition (\ref{E:DecompPara}) of $\Psi_\pi$ we have
\[
    m\left(\eta,a+2k;\pi\right) = m_I\left(\eta,a+2k;\pi\right) + m_J\left(\eta,a+2k;\pi\right).
\]
Similarly, by the decomposition (\ref{E:RelevantDecompPara}) of $\Psi_\sigma$, we have
\[
    m\left(\eta,a+2k+1;\sigma\right) = m_I\left(\eta,a+2k;\pi\right) + m_J\left(\eta,a+2k+2;\pi\right).
\]
Therefore
\begin{align*}
    \Lambda\left(\eta,a;\pi,\sigma\right) &= m\left(\eta,a;\pi\right) - m\left(\eta,a+1;\sigma\right) + m\left(\eta,a+2;\pi\right) - m\left(\eta,a+3;\sigma\right) + \cdots \\
    & = m_J\left(\eta,a;\pi\right)
\end{align*}
is non-negative. One can show that $\Lambda\left(\eta,a;\sigma,\pi\right)$ is also non-negative in the totally same way.

\vskip 5pt

$\bullet$ If $\eta\in\Para_{cmpl}$, i.e. $\eta = \scrL\left(\delta,s\right)$ is the symbol attached to a complementary series, for any non-negative integer $k$ by the decomposition (\ref{E:DecompPara}) of $\Psi_\pi$ we have
\[
    m\left(D^k\left(\eta\right),a+k;\pi\right) = m_I\left(D^k\left(\eta\right),a+k;\pi\right) + m_J\left(D^k\left(\eta\right),a+k;\pi\right) + m_K\left(D^k\left(\eta\right),a+k;\pi\right).
\]
Similarly, by the decomposition (\ref{E:RelevantDecompPara}) of $\Psi_\sigma$, we have
\begin{align*}
    m\left(D^k\left(\eta\right),a+k+1;\sigma\right) = m_I&\left(D^k\left(\eta\right),a+k;\pi\right) + m_J\left(D^{k+2}\left(\eta\right),a+k+2;\pi\right)\\
    &+ m_K\left(D^{k+1}\left(\eta\right),a+k+1;\pi\right).
\end{align*}
Here we have made use of the fact that the operator $D$ is involutive. Therefore
\begin{align*}
    \Lambda\left(\eta,a;\pi,\sigma\right) &= \Big(m\left(\eta,a;\pi\right) + m\left(D\left(\eta\right),a+1;\pi\right) + m\left(\eta,a+2;\pi\right) + \cdots \Big) \\
     & \qquad\qquad- \Big(m\left(\eta,a+1;\sigma\right) + m\left(D\left(\eta\right),a+2;\sigma\right) + m\left(\eta,a+3;\sigma\right) + \cdots \Big)\\
     &= m_J\left(\eta,a;\pi\right)+m_K\left(\eta,a;\pi\right) + m_J\left(D\left(\eta\right),a+1;\pi\right)
\end{align*}
is non-negative. One can show that $\Lambda\left(\eta,a;\sigma,\pi\right)$ is also non-negative in the totally same way.

\vskip 5pt

Next we prove that if these alternating sums $\Lambda\left(\eta,a;\pi,\sigma\right)$ and $\Lambda\left(\eta,a;\sigma,\pi\right)$ are non-negative, then $\pi$ and $\sigma$ are relevant. Suppose that the parameter of $\pi$ has a decomposition as in (\ref{E:DecompPara}), we need to define a partition $\left\{1,2,\cdots, r\right\} = I \sqcup J\sqcup K$ such that (\ref{E:RelevantDecompPara}) holds. We shall do it inductively. Without loss of generality, we may assume that
\[
    d_1\geq d_2 \geq \cdots \geq d_r.
\]
There are two cases.

\vskip 5pt

$\bullet$ \textit{Case 1}: There exists some $\eta'\in\Para$ and $d'>d_1$, such that $\eta'\boxtimes S_{d'}\subset \Psi_\sigma$. Consider the alternating sum
\[
    \Lambda\left(\eta',d'-1;\pi,\sigma\right) = m\left(\eta',d'-1;\pi\right) - m\left(\eta',d';\sigma\right) + \cdots
\]
Since, by our assumption, $\Lambda\left(\eta,d'-1;\pi,\sigma\right)$ is non-negative, one can see immediately that there is some $i\in\left\{1,2,\cdots,r\right\}$, such that $\eta' = \scrL_i$ and $d' =d_i+1$. Without loss of generality, we may assume that $i=1$. Let $\pi_0$ be the irreducible unitary representation corresponding to the parameter
\[
    \Psi_\pi-\scrL_1\boxtimes S_{d_1} = \scrL_2\boxtimes S_{d_2} + \cdots + \scrL_r\boxtimes S_{d_r},
\]
and $\sigma_0$ the irreducible unitary representation corresponding to the parameter
\[
    \Psi_\sigma-\scrL_1\boxtimes S_{d_1+1}.
\]
From these definitions, one can check that alternating sums $\Lambda\left(\eta,a;\pi_0,\sigma_0\right)$ and $\Lambda\left(\eta,a;\sigma_0,\pi_0\right)$ are all non-negative. Indeed, on the one hand, we always have
\[
    \Lambda\left(\eta,a;\pi_0,\sigma_0\right) = \Lambda\left(\eta,a;\pi,\sigma\right).
\]
On the other hand, 
\[
    \Lambda\left(\eta,a;\sigma_0,\pi_0\right) = \Lambda\left(\eta,a;\sigma,\pi\right)
\]
unless: 
\begin{itemize}
    \item[-] $\eta = \scrL_1$ and $a=d_1+1$, in which case one does know that
\[
    \Lambda\left(\scrL_1,d_1+1;\sigma_0,\pi_0\right) = m\left(\scrL_1,d_1+1;\sigma_0\right)
\]
is non-negative;

\vskip 5pt

\item[-] $\scrL_1\in\Para_{cmpl}$, $\eta = D\left(\scrL_1\right)$ and $a=d_1$, in which case one does know that
\[
    \Lambda\left(\scrL_1,d_1+1;\sigma_0,\pi_0\right) = m\left(D\left(\scrL_1\right),d_1;\sigma_0\right) + m\left(\scrL_1,d_1+1;\sigma_0\right)
\]
is non-negative.
\end{itemize}
\vskip 5pt
Therefore, by the induction hypothesis $\pi_0$ and $\sigma_0$ are relevant, i.e. there is a partition $\left\{2,\cdots, r\right\} = I' \sqcup J\sqcup K$, such that
\[
    \Psi_{\sigma_0} = \sum_{i\in I'}\scrL_i\boxtimes S_{d_i+1} +\sum_{j\in J}\scrL_j\boxtimes S_{d_j-1} + \sum_{k\in K} D\left(\scrL_k\right) \boxtimes S_{d_k} + \Psi_0
\]
for a generic parameter $\Psi_0$. Put $I = \left\{1\right\} \cup I'$. Then $\left\{1,2,\cdots, r\right\} = I \sqcup J\sqcup K$ is a desired partition.

\vskip 5pt

$\bullet$ \textit{Case 2}: For any $\eta'\in\Para$ and $d'>d_1$, we have $\eta'\boxtimes S_{d'}\not\subset \Psi_\sigma$. If $d_1=1$, then both $\pi$ and $\sigma$ are generic, hence they are relevant. So we may assume that $d_1>1$. Consider the alternating sum
\[
    \Lambda\left(\scrL_1,d_1-1;\sigma,\pi\right) = m\left(\scrL_1,d_1-1;\sigma\right) - m\left(\scrL_1,d_1;\pi\right) + \cdots
\]
Since by our assumption $\Lambda\left(\scrL_1,d_1-1;\sigma,\pi\right)$ is non-negative, one can see immediately that either

\begin{itemize}
    \item[-] $\scrL_1\in\Para_{cmpl}$, and $D\left(\scrL_1\right)\boxtimes S_{d_1} \subset \Psi_\sigma$; or

    \vskip 5pt

    \item[-] $\scrL_1\boxtimes S_{d_1-1}\subset \Psi_\sigma$.
\end{itemize}
\vskip 5pt
Similar to the argument of Case 1, let $\pi_0$ be the irreducible unitary representation corresponding to the parameter 
\[
    \Psi_\pi-\scrL_1\boxtimes S_{d_1} = \scrL_2\boxtimes S_{d_2} + \cdots + \scrL_r\boxtimes S_{d_r},
\]
and $\sigma_0$ the irreducible unitary representation corresponding to the parameter
\[
    \begin{cases}
        \Psi_\sigma-D\left(\scrL_1\right)\boxtimes S_{d_1}, \quad & \textit{if }\scrL_1\in\Para_{cmpl}, \textit{ and } D\left(\scrL_1\right)\boxtimes S_{d_1}\subset \Psi_\sigma;
        \vspace{0.5em}\\

        \Psi_\sigma-\scrL_1\boxtimes S_{d_1-1}, \quad & \textit{otherwise}.
    \end{cases}
\]
One can check that alternating sums $\Lambda\left(\eta,a;\pi_0,\sigma_0\right)$ and $\Lambda\left(\eta,a;\sigma_0,\pi_0\right)$ are all non-negative. Hence, by the induction hypothesis, we know that $\pi_0$ and $\sigma_0$ are relevant, from which we conclude that $\pi$ and $\sigma$ are also relevant.

\end{proof}


\vskip 10pt

\section{Orbit analysis and higher corank models} 
\label{sec:orbit_analysis_and_higher_corank_models}

\subsection{Orbit analysis on Grassmannian} 
\label{sub:orbit_analysis_on_grassmannian_rank_1_case}

We first investigate how a parabolic induced representation $\rho\times\tau$ of $G_{n+1}$ decomposes when restricted to the mirabolic subgroup $M_{n+1}$ using Borel's lemma (Proposition \ref{P:BorelLemma}), where $\rho$ is an essential discrete series representation of $G_k$. Let $P_{k,n+1-k}$ be the standard parabolic subgroup of $G_{n+1}$ with Levi component $G_k\times G_{n+1-k}$. Then we need to do the orbit analysis on the Grassmannian
\[
    P_{k,n+1-k}\backslash G_{n+1} = \Gr\left(k,n+1\right).
\]
Let $\calE_{\rho,\tau}$ be the tempered vector bundle over $\calX = \Gr\left(k,n+1\right)$ corresponding to $\rho\times\tau$. There are exactly two orbits of $M_{n+1}$ on $\calX$.

\vskip 5pt

$\bullet$ The closed orbit $\calZ = \Gr\left(k,n\right)$ consisting of $k$-dimensional subspaces of $F^n$. The action of $M_{n+1}$ on this orbit descends to $G_n$; the stabilizer in $G_n$ of a representative in this orbit is $P_{k,n-k}$. By some elementary computations we know that the conormal bundle of $\calZ$ in $\calX$ is
    \[
        N_{\calZ|\calX}^\vee \simeq G_n\times_{P_{k,n-k}}\std_k,
    \]
    where as a vector space $\std_k = F^k$, the action of $P_{k,n-k}$ on $F^k$ factors through the Levi component $G_k\times G_{n-k}$, $G_k$ acts on $F^k$ by the standard representation, and $G_{n-k}$ acts trivially. Applying Borel's lemma, the contribution of the closed orbit (and derivatives) $\Gamma_{\calZ}\left(\calX, \calE_{\rho,\tau}\right)$ has a filtration, with each graded piece
    \[
        \left(\Sym^j\std_{k,\C}\otimes\rho\right)|{\det}_k|^{\frac{1}{2}}\times\left(\tau~\big|_{G_{n-k}}\right)
    \]
    as representations of $G_n$, where $j=0,1,2,\cdots$.

    \vskip 5pt

$\bullet$ The open orbit $\calX\backslash\calZ$ consisting of $k$-dimensional subspaces not contained in $F^n$. Let $S_k$ be the stabilizer in $M_{n+1}$ of a representative in this orbit. It fits into an exact sequence
    \[
        1 \lra N_{k-1,n+1-k} \lra S_k \lra M_k\times G_{n+1-k} \lra 1,
    \]
    where $N_{k-1,n+1-k}$ is the unipotent radical of $P_{k-1,n+1-k}$. The contribution of the open orbit (as a representation of $M_{n+1}$) is
    \[
        \leftindex^u\ind_{S_k}^{M_{n+1}}\left(\rho|{\det}_k|^{\frac{n+1-k}{2}}~\big|_{M_k}\boxtimes\tau|{\det}_{n+1-k}|^{-\frac{k}{2}}\right),
    \]
    where $\rho|{\det}_k|^{\frac{n+1-k}{2}}~\big|_{M_k}\boxtimes\tau|{\det}_{n+1-k}|^{-\frac{k}{2}}$ is a representation of $M_k\times G_{n+1-k}$, and we regard it as a representation of $S_k$ by inflation.  

\vskip 5pt

We give two useful applications of discussions in this section. The first application concerns the so-called \textit{Fourier--Jacobi models}, namely, the Hom space
\[
    \Hom_{G_n}\left(\pi\otimes \omega_n,\sigma\right)
\]
for two representations $\pi$ and $\sigma$ of $G_n$.

\begin{Lem}\label{L:FJToBessel}
Let $\pi$ and $\sigma$ be finite length representations of $G_n$. Then for all but at most countably many $s\in \sqrt{-1}\cdot\R$, there is a natural isomorphism 
\[
    \Hom_{G_n}\left(|\cdot|^s\times\pi,\sigma\right) \simeq \Hom_{G_n}\left(\pi\otimes \omega_n,\sigma\right).
\]   
\end{Lem}

\vskip 5pt

\begin{proof}
Apply the discussion above to $|\cdot|^s\times\pi$. By Lemma \ref{L:inftesimalArguement}, we know that for all but at most countably many $s\in \sqrt{-1}\cdot\R$, we have
\[
    \SH_i\left(G_n, \left(\Sym^j\std_{1,\C}\otimes|\cdot|^{s+\frac{1}{2}}\times\left(\pi~\big|_{G_{n-1}}\right)\right)\otimes\sigma^\vee\right) = 0
\]
for all $i,j = 0,1,2,\cdots$. This implies that for all but at most countably many $s\in\sqrt{-1}\cdot\R$, we have
\[
    \Hom_{G_n}\left(|\cdot|^s\times\pi,\sigma\right) \simeq \Hom_{G_n}\left(\leftindex^u\ind_{G_n}^{M_{n+1}}\pi|{\det}_{n}|^{-\frac{1}{2}}~\big|_{G_n},\sigma\right).
\]  
Note that $\leftindex^u\ind_{G_n}^{M_{n+1}}\pi|{\det}_{n}|^{-\frac{1}{2}}~\big|_{G_n}\simeq \pi\otimes\omega_n$, hence we are done.

\end{proof}

\vskip 5pt

To give the second application, we specialize to the case that $F=\R$ and $n=k=1$. Let $\delta$ be an irreducible unitary discrete series representation of $G_2$. We would like to investigate how $\delta$ decomposes when restricted to $M_2$. It is well-known that discrete series of $G_2$ can be constructed as the quotient of a principal series by a finite-dimensional subrepresentation. To be more precise, there exists a unitary character $\chi$ of $G_1$ and a positive integer $m\in\Z_{>0}$, such that $\delta$ fits into an exact sequence
\begin{equation}\label{E:dsExactSeq}
    0 \lra \Pi_f \lra \chi\,{\rm sgn}^{m-1}|\cdot|^{-\frac{m}{2}}\times\chi|\cdot|^{\frac{m}{2}} \lra \delta \lra 0,
\end{equation}
where $\Pi_f=\chi\,{\rm sgn}^{m-1}|\cdot|^{\frac{1-m}{2}}{\rm Sym}^{m-1}\C^2$ is an $m$-dimensional representation of $G_2$. Applying the discussion in this section to $\chi\,{\rm sgn}^{m-1}|\cdot|^{-\frac{m}{2}}\times\chi|\cdot|^{\frac{m}{2}}$, one knows that:

\begin{itemize}
    \item as a representation of $M_2$, $\chi\,{\rm sgn}^{m-1}|\cdot|^{-\frac{m}{2}}\times\chi|\cdot|^{\frac{m}{2}}$ has a subrepresentation
    \[
        \Gamma_{op} \simeq \leftindex^u\ind_{G_1}^{M_2}\chi|\cdot|^{\frac{m-1}{2}};
    \]

    \vskip 5pt

    \item the quotient of $\chi\,{\rm sgn}^{m-1}|\cdot|^{-\frac{m}{2}}\times\chi|\cdot|^{\frac{m}{2}}$ by $\Gamma_{op}$ has a descending filtration, with each graded pieces isomorphism to the inflation of
    \[
        \chi\,{\rm sgn}^{j+m-1}|\cdot|^{j+\frac{1-m}{2}}
    \]
    to $M_2$, where $j=0,1,2,\cdots$.
\end{itemize}
\vskip 5pt
On the other hand, easy to see that the restriction of $\Pi_f$ to $M_2$ has a filtration, with each graded piece isomorphic to the inflation of
\[
        \chi\,{\rm sgn}^{j+m-1}|\cdot|^{j+\frac{1-m}{2}}
\]
to $M_2$, where $j=0,1,2,\cdots,m-1$. We can conclude that:

\begin{Lem}\label{L:Res.d.s.To.Mirabolic}
In the context of the above discussion (i.e. $F=\R$ and $n=k=1$), let $\delta$ be the discrete series representation of $G_2$ which fits into the exact sequence (\ref{E:dsExactSeq}). Then there is an exact sequence of $M_2$-representations
\[
    0 \lra \Gamma_{op} \lra \delta~\big|_{M_2} \lra \Gamma_\delta \lra 0,
\] 
where $\Gamma_{op}$ is defined as above, and $\Gamma_\delta$ has a descending filtration, with each graded piece isomorphic to the inflation of
    \[
        \chi\,{\rm sgn}^{j}|\cdot|^{j+\frac{m-1}{2}}
    \]
    to $M_2$, where $j=1,2,\cdots$.
\end{Lem}

\vskip 5pt

\begin{proof}
By our discussions, to show this lemma it suffices to show that $\Gamma_{op}\cap \Pi_f = 0$. This can be proved as in \cite[Prop. 5.2.4]{chen2021local}. 

\end{proof}

\subsection{Rankin--Selberg models of higher corank} 
\label{sub:rankin_selberg_models_of_higher_corank}

Now back to the general case. In the proof of our main result, some more general models will be involved. Let $m\leq n$ be two non-negative integers. We set $Q_{m,n}$ to be the standard parabolic subgroup of $G_{n+1}$ with Levi component
\[
    G_{m+1}\times\underbrace{G_1\times\cdots\times G_1}_{n-m\textit{-times}},
\]
and denote its unipotent radical by $U_{m,n}$. As before, we regard $G_m$ as a subgroup of $G_{m+1}$. The corank $n+1-m$ \textit{Rankin--Selberg subgroup} of $G_{n+1}$ is defined to be
\[
    RS_{m,n} = G_m\cdot U_{m,n}.
\]
Fix a non-trivial additive character $\psi_F$ of $F$; we can define a generic character
\[
    \psi_{m,n}: U_{m,n} \lra \C, \quad u \longmapsto \psi_F\left(\sum_{i=1}^{n-m}u_{m+i,m+i+1}\right),
\]
where $u_{a,b}$ is the $\left(a,b\right)$-entry of $u\in U_{m,n}$. We are interested in determining the space
\[
    \Hom_{RS_{m,n}}\left(\pi,\sigma\boxtimes \psi_{m,n}\right)
\]
for irreducible representations $\pi$ of $G_{n+1}$ and $\sigma$ of $G_m$. These models also enjoy the multiplicity one property \cite{MR3384460}. It is worth noting that: 

\begin{itemize}
    \item when $m=n$, by definition $RS_{m,n}=G_n$, this recovers the corank $1$ model considered at the beginning of this paper;

    \vskip 5pt

    \item when $m=0$, by definition $RS_{m,n}$ is the unipotent radical of the Borel subgroup, this recovers the usual Whittaker model.
\end{itemize}

\vskip 5pt

Next we follow \cite[Sect. 6]{MR3384460} to relate these general models to the corank $1$ case. 

\begin{Prop}\label{P:ReductionToCorank1}
Let $\pi$ be a finite length representation of $G_{n+1}$, and $\sigma$ a finite length representation of $G_m$. Then there exists an exceptional subset ${\rm Ex}\subset\C^{n+1-m}$, which is a union of countably many Zariski closed proper subsets of $\C^{n+1-m}$, such that for all $\left(s_1, s_2, \cdots, s_{n+1-m}\right) \in \C^{n+1-m}\,\backslash\, {\rm Ex}$, there is a natural isomorphism
\[
    \Hom_{G_{n+1}}\left(\sigma^\vee\times I\left(s_1,s_2,\cdots,s_{n+1-m}\right)\otimes\omega_{n+1}, \pi^\vee\right) \simeq \Hom_{RS_{m,n}}\left(\pi,\sigma\boxtimes \psi_{m,n}\right).
\]
Here $I\left(s_1,s_2,\cdots,s_{n+1-m}\right)\coloneqq|\cdot|^{s_1}\times|\cdot|^{s_2}\times\cdots\times|\cdot|^{s_{n+1-m}}$.
\end{Prop}

\vskip 5pt

\begin{proof}
A similar statement in the non-Archimedean setting has been established in \cite[Sect. 6]{MR3384460}; the argument there works for the Archimedean case with some mild modifications. For the convenience of readers, we briefly sketch the proof here.

\vskip 5pt

Recall that $\omega_{n+1}$ is realized on the space of Schwartz functions on $F^{n+1}$. Consider the subspace $\omega_{n+1}^0$ of $\omega_{n+1}$ consisting of Schwartz functions supported on $F^{n+1}\backslash\left\{0\right\}$, it induces a filtration of the Weil representation
\[
    0 \lra \omega_{n+1}^0 \lra \omega_{n+1} \lra \omega_{n+1}/\omega_{n+1}^0 \lra 0.
\] 
According to Borel's lemma (Proposition \ref{P:BorelLemma}), $\omega_{n+1}/\omega_{n+1}^0$ has a filtration, with each graded piece
\[
    \left(\Sym^j\std_{n+1,\C}\right)\otimes |{\det}_{n+1}|^{-\frac{1}{2}},
\]
where $j=0,1,\cdots$. By Lemma \ref{L:inftesimalArguement}, one knows that there exists an exceptional subset ${\rm Ex}_\alpha\subset\C^{n+1-m}$, which is a union of countably many Zariski closed proper subsets, such that for all $\left(s_1, s_2, \cdots, s_{n-m+1}\right) \in \C^{n+1-m}\,\backslash\, {\rm Ex}_\alpha$, one has
\[
    \SH_i\left({G_{n+1}},\left(\sigma^\vee\times I\left(s_1,s_2,\cdots\right)\otimes\left(\Sym^j\std_{n+1,\C}\right)\otimes |{\det}_{n+1}|^{-\frac{1}{2}}\right)\otimes\pi\right) = 0
\]
for all $i,j= 0,1,2,\cdots$. This implies that
\[
    \Hom_{G_{n+1}}\left(\sigma^\vee\times I\left(s_1,s_2,\cdots\right)\otimes\omega_{n+1},\pi^\vee\right) 
    \simeq \Hom_{G_{n+1}}\left(\sigma^\vee\times I\left(s_1,s_2\cdots\right)\otimes\omega_{n+1}^0, \pi^\vee\right).
\]
\vskip 5pt

Now note that 
\[
    \omega_{n+1}^0 \simeq \left(\leftindex^u\ind_{M_{n+1}^t}^{G_{n+1}} \C\right)\otimes |{\det}_{n+1}|^{-\frac{1}{2}},
\]
where $M_{n+1}^t$ is the transpose of the mirabolic subgroup $M_{n+1}$, or equivalently, the subgroup of $G_{n+1}$ stabilizing
\[
    e_{n+1} \coloneqq \left(0,\cdots,0,1\right)^t\in F^{n+1}
\]
under the standard action. Hence we have
\[
    \sigma^\vee\times I\left(s_1,s_2,\cdots\right)\otimes\omega_{n+1}^0 \simeq |{\det}_{n+1}|^{-\frac{1}{2}}\otimes\leftindex^u\ind_{M_{n+1}^t}^{G_{n+1}} \left(\sigma^\vee\times I\left(s_1,s_2,\cdots\right)~\big|_{M_{n+1}^t}\right).
\]
To describe the representation on the right hand side of above, once again we appeal to the Mackey theory. Let $P_{m,n+1-m}$ be the standard parabolic subgroup of $G_{n+1}$ with Levi component $G_m\times G_{n+1-m}$, and $\calE_{\sigma,\underline{s}}$ the tempered vector bundle over the Grassmannian $\calX = P_{m,n+1-m}\backslash G_{n+1}$ corresponding to $\sigma^\vee\times I\left(s_1,s_2,\cdots\right)$. Similar to discussions in Section \ref{sub:orbit_analysis_on_grassmannian_rank_1_case}, there are exactly two orbits of $M_{n+1}^t$ on the $\calX$.

\vskip 5pt

$\bullet$ The closed orbit $\calZ$ consisting of $m$-dimensional subspaces $X\subset F^{n+1}$ with the property that $e_{n+1}\in X$. 
    Let $T_m$ be the stabilizer in $M_{n+1}^t$ of a representative in this orbit. It fits into an exact sequence
    \[
        1 \lra N_{m,n+1-m} \lra T_m \lra M_m^t\times G_{n+1-m} \lra 1,
    \]
    where $N_{m,n+1-m}$ is the unipotent radical of $P_{m,n+1-m}$. By some elementary computations, we know that the conormal bundle of $\calZ$ in $\calX$ is
    \[
        N_{\calZ|\calX}^\vee \simeq M_{n+1}^t\times_{T_m}\std_{n+1-m}^\vee,
    \]
    where the action of $T_m$ on $\std_{n+1-m}^\vee$ descends to $M_m^t\times G_{n+1-m}$, $G_{n+1-m}$ acts as $\std_{n+1-m}^\vee$, and $M_m^t$ acts trivially. Applying Borel's lemma, the contribution of the closed orbit (and derivatives) $\Gamma_\calZ\left(\calX,\calE_{\sigma,\underline{s}}\right)$ has a filtration, with each graded piece  
    \[
        \leftindex^u\ind_{T_m}^{M^t_{n+1}}\left(\sigma^\vee|{\det}|^{\frac{n+1-m}{2}}~\big|_{M_m^t}\boxtimes\,\Sym^j\std_{n+1-m,\C}^\vee\otimes I\left(s_1,s_2,\cdots\right)|{\det}|^{-\frac{m}{2}}\right),
    \]
    where $\sigma^\vee|{\det}|^{\frac{n+1-m}{2}}~\big|_{M_m^t}\boxtimes\,\Sym^j\std^\vee\otimes I\left(s_1,\cdots\right)|{\det}|^{-\frac{m}{2}}$ is a representation of $M_m^t\times G_{n+1-m}$, we regard it as a representation of $T_m$ by inflation; and $j=0,1,2,\cdots$. After applying the functor $|{\det}_{n+1}|^{-\frac{1}{2}}\otimes\leftindex^u\ind_{M_{n+1}^t}^{G_{n+1}}$, each graded piece becomes to
    \[
        |{\det}_{n+1}|^{-\frac{1}{2}}\otimes\leftindex^u\ind_{T_m}^{G_{n+1}}\left(\sigma^\vee|{\det}|^{\frac{n+1-m}{2}}~\big|_{M_m^t}\boxtimes\,\Sym^j\std_{n+1-m,\C}^\vee\otimes I\left(s_1,\cdots\right)|{\det}|^{-\frac{m}{2}}\right)
    \]
    \[
        \simeq \left(\sigma^\vee\otimes\omega_m\right)\times \left(\Sym^j\std_{n+1-m,\C}^\vee\otimes I\left(s_1,\cdots\right)|{\det}|^{-\frac{1}{2}}\right),
    \]
    where $j=0,1,2,\cdots$. 

    \vskip 5pt

$\bullet$ The open orbit $\calX\backslash\calZ$ consisting of $m$-dimensional subspaces $X\subset F^{n+1}$ with the property that $e_{n+1}\notin X$. Let $S_m$ be the stabilizer in $M_{n+1}^t$ of a representative in this orbit. It fits into an exact sequence
    \[
        1 \lra N^0_{m,n+1-m} \lra S_m \lra G_m\times \leftindex^wM^t_{n+1-m} \lra 1,
    \]
    where $N_{m,n+1-m}$ is the unipotent radical of $P_{m,n+1-m}$, $N^0_{m,n+1-m}$ is the subgroup of $N_{m,n+1-m}$ consisting of elements of the form 
\[
        \left(\begin{array}{ccc}
                            I_m & 0 & X\\
                            ~ & 1 & ~\\
                            ~ & ~ & I_{n-m}
                    \end{array}\right), \qquad X\in {\rm Mat}_{m,n-m}\left(F\right);
\]
and $\leftindex^wM^t_{n+1-m}$ is the subgroup of $G_{n+1-m}$ consisting of elements of the form
\[
    \left(\begin{array}{cc}
                            1 & v\\
                            ~ & g
                    \end{array}\right), \qquad g\in G_{n-m}, v\in F^{n-m}.
\]
One can easily check that $\leftindex^wM^t_{n+1-m}$ is a conjugation of $M^t_{n+1-m}$. The contribution of the open orbit $\Gamma\left(\calX\backslash\calZ,\calE_{\sigma,\underline{s}}\right)$ is
    \[
        \leftindex^u\ind_{S_m}^{M^t_{n+1}}\left(\sigma^\vee|{\det}|^{\frac{n+1-m}{2}}\boxtimes \,I\left(s_1,s_2,\cdots\right)|{\det}|^{-\frac{m}{2}}~\big|_{\leftindex^wM^t_{n+1-m}}\right),
    \]
    where $\sigma^\vee|{\det}|^{\frac{n+1-m}{2}}\boxtimes \,I\left(s_1,s_2,\cdots\right)|{\det}|^{-\frac{m}{2}}~\big|_{\leftindex^wM^t_{n+1-m}}$ is a representation of $G_m\times \leftindex^wM^t_{n+1-m}$, and we regard it as a representation of $S_m$ by inflation. After applying the functor $|{\det}_{n+1}|^{-\frac{1}{2}}\otimes\leftindex^u\ind_{M_{n+1}^t}^{G_{n+1}}$, it becomes to
    \[
        |{\det}_{n+1}|^{-\frac{1}{2}}\otimes\leftindex^u\ind_{S_m}^{G_{n+1}}\left(\sigma^\vee|{\det}|^{\frac{n+1-m}{2}}\boxtimes \,I\left(s_1,s_2,\cdots\right)|{\det}|^{-\frac{m}{2}}~\big|_{\leftindex^wM^t_{n+1-m}}\right).
    \]

\vskip 5pt

Similarly, by using the infinitesimal character argument (Lemma \ref{L:inftesimalArguement}), there exists an exceptional subset ${\rm Ex}_\beta\subset \C^{n+1-m}$, which is a union of countably many Zariski closed proper subsets, such that for all $\left(s_1, \cdots, s_{n-m+1}\right) \in \C^{n+1-m}\,\backslash\, {\rm Ex}_\beta$, one has
\[
    \SH_i\left(G_{n+1}, \left(\sigma^\vee\otimes\omega_m\right)\times \left(\Sym^j\std_{n+1-m,\C}^\vee\otimes I\left(s_1,\cdots\right)|{\det}|^{-\frac{1}{2}}\right)\otimes\pi\right) = 0
\]
for all $i,j= 0,1,2,\cdots$. This implies that
\[
    \Hom_{G_{n+1}}\left(\sigma^\vee\times I\left(s_1,s_2,\cdots\right)\otimes\omega_{n+1}^0, \pi^\vee\right) 
\]
\[
    \simeq \Hom_{G_{n+1}}\left(|{\det}_{n+1}|^{-\frac{1}{2}}\otimes\leftindex^u\ind_{S_m}^{G_{n+1}}\left(\sigma^\vee|{\det}|^{\frac{n+1-m}{2}}\boxtimes \,I\left(s_1,s_2,\cdots\right)|{\det}|^{-\frac{m}{2}}~\big|_{\leftindex^wM^t_{n+1-m}}\right), \pi^\vee\right).
\]

\vskip 5pt

Finally, we need to describe the restriction of $I\left(s_1,s_2,\cdots\right)$ to the subgroup $\leftindex^wM^t_{n+1-m}$. The restriction of a principal series representation to a mirabolic subgroup has been studied by Xue \cite[Prop. 5.1]{xue2020besselShomology} when $F=\C$, and by the first author \cite[Prop. 5.2.5]{chen2021local} when $F=\R$ using Borel's lemma. 
Consider the \textit{MVW-involution}
\[
    c: G_{n+1-m} \lra G_{n+1-m}, \quad g \longmapsto w\cdot\leftindex^{t}g^{-1}\cdot w^{-1},
\]
where $w\in G_{n+1-m}$ is a suitably chosen element. Note that the MVW-involution of $G_{n+1-m}$ takes principal series representations to principal series representations, and takes the mirabolic subgroup $M_{n+1-m}$ to $\leftindex^wM^t_{n+1-m}$. Hence we can combine \cite[Prop. 5.1]{xue2020besselShomology}, \cite[Prop. 5.2.5]{chen2021local} with the MVW-involution to describe $I\left(s_1,s_2,\cdots\right)~\big|_{\leftindex^wM^t_{n+1-m}}$. We summarize the result as follows.

\vskip 5pt

$\bullet$ The representation $I\left(s_1,s_2,\cdots\right)~\big|_{\leftindex^wM^t_{n+1-m}}$ has a subrepresentation isomorphic to
    \[
        \leftindex^u\ind_{~U_{n+1-m}}^{\leftindex^wM^t_{n+1-m}}\psi_{n+1-m}^{-1},
    \]
    where $U_{n+1-m}$ is the subgroup of $\leftindex^wM^t_{n+1-m}$ consisting of upper triangular unipotent matrices, and $\psi_{n+1-m}$ is a generic character of $U_{n+1-m}$.

\vskip 5pt

$\bullet$ The quotient 
\[
    I\left(s_1,s_2,\cdots\right)~\big|_{\leftindex^wM^t_{n+1-m}} \big/ \leftindex^u\ind_{~U_{n+1-m}}^{\leftindex^wM^t_{n+1-m}}\psi_{n+1-m}^{-1}
\]
admits a $\leftindex^wM^t_{n+1-m}$-stable filtration whose graded pieces are of the form
\[
    \leftindex^u\ind_{~Q_{a,b,c}}^{\leftindex^wM^t_{n+1-m}} \tau_a\boxtimes\tau_b\boxtimes\tau_c,
\]
where:

\begin{itemize}
    \item[-] $Q_{a,b,c}$ is the intersection of $P_{a,b,c}$ with $\leftindex^wM^t_{n+1-m}$, where $a,b,c$ are non-negative integers such that $a+b+c= n+1-m$ and $b+c>0$, $P_{a,b,c}$ is the standard parabolic subgroup of $G_{n+1-m}$with Levi component $G_a\times G_b\times G_c$; 

    \vskip 5pt

    \item[-] the group $Q_{a,b,c}$ fits into an exact sequence
    \[
     1 \lra N_{a,b,c} \lra Q_{a,b,c} \lra \leftindex^wM^t_{a}\times G_b\times G_c \lra 1,
    \]
    where $N_{a,b,c}$ is the unipotent radical of $P_{a,b,c}$;

    \vskip 5pt

    \item[-] $\tau_a$ is a representation of $\leftindex^wM^t_{a}$, and is isomorphic to $\leftindex^u\ind_{~U_{a}}^{\leftindex^wM^t_{a}}\psi_{a}^{-1}$;

    \vskip 5pt

    \item[-] $\tau_b$ is a representation of $G_b$ which is of the form $\tau'_b\otimes \rho$, where $\tau'_b$ is a principal series representation of the same form as $\tau_c$ described below, and $\rho$ is a finite dimensional representation of $G_b$;

    \vskip 5pt

    \item[-] $\tau_c$ is a principal series representation of $G_c$ which is of the form
    \[
        \chi_1|\cdot|^{s_{i_1}+k_1}\times\chi_2|\cdot|^{s_{i_2}+k_2}\times\cdots\times\chi_c|\cdot|^{s_{i_c}+k_c},
    \]
    where $1\leq i_1,i_2,\cdots,i_c\leq n+1-m$, $\chi_1,\chi_2,\cdots,\chi_c$ are (conjugate) self-dual characters of $F^\times$, and $k_1,k_2,\cdots,k_c\in \Z$;

    \vskip 5pt

    \item[-] $\tau_a\boxtimes\tau_b\boxtimes \tau_c$ is a representation of $\leftindex^wM^t_{a}\times G_b\times G_c$, regarded as a representation of $Q_{a,b,c}$ by inflation. 
\end{itemize}
\vskip 5pt

Easy to check that
\[
    |{\det}_{n+1}|^{-\frac{1}{2}}\otimes\leftindex^u\ind_{S_m}^{G_{n+1}}\left(\sigma^\vee|{\det}|^{\frac{n+1-m}{2}}\boxtimes \,\left(\leftindex^u\ind_{~Q_{a,b,c}}^{\leftindex^wM^t_{n+1-m}} \tau_a\boxtimes\tau_b\boxtimes\tau_c\right)|{\det}_{n+1-m}|^{-\frac{m}{2}}\right)
\]
\[
    \simeq \leftindex^u\ind_{RS_{m,m+a-1}}^{G_{m+a}} \left(\sigma^\vee|\det|^{\frac{a-1}{2}}\boxtimes\psi^{-1}_{m,m+a-1}\right)\times \tau_b|\det|^{\ell_b}\times\tau_c|\det|^{\ell_c}
\]
for some half integer $\ell_b,\ell_c\in\frac{1}{2}\Z$. Once again, by using the infinitesimal characters argument (Lemma \ref{L:inftesimalArguement}), one can show that there exists an exceptional subset ${\rm Ex}_\gamma\subset\C^{n+1-m}$, which is a union of countably many Zariski closed proper subsets, such that for all $\left(s_1, s_2, \cdots, s_{n-m+1}\right) \in \C^{n+1-m}\,\backslash\, {\rm Ex}_\gamma$, one has
\[
    \SH_i\left(G_{n+1}, \left(\leftindex^u\ind_{RS_{m,m+a-1}}^{G_{m+a}} \left(\sigma^\vee|\det|^{\frac{a-1}{2}}\boxtimes\psi^{-1}_{m,m+a-1}\right)\times \tau_b|\det|^{\ell_b}\times\tau_c|\det|^{\ell_c}\right)\otimes \pi\right) = 0
\]
for all $i=0,1,2,\cdots$ and all $\leftindex^u\ind_{~Q_{a,b,c}}^{\leftindex^wM^t_{n+1-m}} \tau_a\boxtimes\tau_b\boxtimes\tau_c$. This implies that
\[
    \Hom_{G_{n+1}}\left(|{\det}_{n+1}|^{-\frac{1}{2}}\otimes\leftindex^u\ind_{S_m}^{G_{n+1}}\left(\sigma^\vee|{\det}|^{\frac{n+1-m}{2}}\boxtimes \,I\left(s_1,s_2,\cdots\right)|{\det}|^{-\frac{m}{2}}~\big|_{\leftindex^wM^t_{n+1-m}}\right), \pi^\vee\right)
\]
\[
    \simeq \Hom_{G_{n+1}}\left(\leftindex^u\ind_{RS_{m,n}}^{G_{n+1}} \left(\sigma^\vee|\det|^{\frac{n-m}{2}}\boxtimes\psi_{m,n}^{-1}\right),\pi^\vee\right).
\]
Let ${\rm Ex} = {\rm Ex}_\alpha\cup{\rm Ex}_\beta\cup{\rm Ex}_\gamma$ and combine all these isomorphisms, we eventually get
\[
    \Hom_{G_{n+1}}\left(\sigma^\vee\times I\left(s_1,s_2,\cdots\right)\otimes\omega_{n+1},\pi^\vee\right)
    \simeq \Hom_{RS_{m,n}}\left(\pi,\sigma\boxtimes \psi_{m,n}\right)
\]
for all $\left(s_1, s_2, \cdots, s_{n+1-m}\right) \in \C^{n+1-m}\,\backslash\, {\rm Ex}$ as desired. This completes the proof.

\end{proof}

\vskip 5pt

Combine Lemma \ref{L:FJToBessel} and Proposition \ref{P:ReductionToCorank1}, we obtain the following consequence.  
\begin{Cor}\label{C:ExchangeVariables}
Let $\pi$ be a finite length representation of $G_{n+1}$, and $\sigma$ a finite length representation of $G_n$. Then there exists a (at most) countable subset ${\rm Ex}\subset\sqrt{-1}\cdot\R$, such that for all $s_1, s_2\in \sqrt{-1}\cdot\R\,\backslash\, {\rm Ex}$, there is a natural isomorphism
\[
    \Hom_{G_n}\left(\pi,\sigma\right) \simeq \Hom_{G_{n+1}}\left(|\cdot|^{s_1}\times\sigma^\vee\times|\cdot|^{s_2},\pi^\vee\right).
\]   
\end{Cor}


\vskip 10pt


\section{Proof of the main theorem} 
\label{sec:proof_of_the_main_result_i_period_implies_relevance}

\subsection{Reduction step: a deformation argument} 
\label{sub:reduction_i_the_easy_step}


We begin the proof of Theorem \ref{T:MainNTGGPGL}(1). Let $\pi$ and $\sigma$ be irreducible unitary representations of $G_{n+1}$ and $G_n$ respectively. Suppose that 
\[
    \Hom_{G_n}\left(\pi,\sigma\right) \neq 0.
\]
We need to show that $\pi$ and $\sigma$ are relevant, and we shall do it by induction on $NT_\pi+NT_\sigma$. The basic case is of course the generic case (i.e. $NT_\pi+NT_\sigma=0$), which is already known. So now assume that at least one of $\pi$ and $\sigma$ is non-generic. The first step is to use Lemma \ref{L:vanishingLem1} repeatedly to reduce $\pi$ and $\sigma$ to ``\textit{more tempered}'' representations. We write
\[
    \Psi_\pi = \scrL_1\boxtimes S_{p_1} + \scrL_2\boxtimes S_{p_2} + \cdots + \scrL_r\boxtimes S_{p_r},
\]
and likewise
\[
    \Psi_\sigma = \scrL'_1\boxtimes S_{q_1} + \scrL'_2\boxtimes S_{q_2} + \cdots + \scrL'_s\boxtimes S_{q_s}.
\]
Firstly we can assume that $p_1$ is (one of) the biggest integers among $\left\{p_1,p_2,\cdots,p_r,q_1,q_2,\cdots,q_s\right\}$. This is because if some $q_j$ is the biggest, then by Corollary \ref{C:ExchangeVariables} we can ``exchange variables'', namely, suitably pick two unitary characters $\xi_1$ and $\xi_2$, such that
\[
    \Hom_{G_{n+1}}\left(\xi_1\times\sigma^\vee\times\xi_2, \pi^\vee\right) \simeq \Hom_{G_n}\left(\pi,\sigma\right).
\]
Since $\pi$ or $\sigma$ is non-generic, we know that $p_1\geq 2$. Let $\rho$ be either a unitary discrete series $\delta_1$ or $\delta_1|\det|^{s_1}$ of a group $G_k$, depending on the parameter $\scrL_1$ is of the form $\scrL\left(\delta_1\right)$ or $\scrL\left(\delta_1,s_1\right)$. Let $\pi_0$ be the irreducible unitary representation corresponding to the parameter
\[
    \begin{cases}
        \Psi_\pi-\scrL_1\boxtimes S_{p_1} + \scrL_1\boxtimes S_{p_1-2}, \quad & \textit{if }\scrL_1=\scrL\left(\delta_1\right);
        \vspace{0.5em}\\

        \Psi_\pi-\scrL_1\boxtimes S_{p_1}+D\left(\scrL_1\right)\boxtimes S_{p_1-1}, \quad & \textit{if }\scrL_1=\scrL\left(\delta_1,s_1\right).
    \end{cases}
\]
Then we have a surjection
\[
    \rho|\cdot|^{\frac{p_1-1}{2}}\times\pi_0\times R\left(\rho\right)|\cdot|^{-\frac{p_1-1}{2}} \lra \pi,
\]
where $R\left(\rho\right)=\delta_1$ if $\rho=\delta_1$ and $R\left(\rho\right)=\delta_1|\det|^{-s_1}$ if $\rho=\delta_1|\det|^{s_1}$. To simplify notations we temporarily set $\Pi\coloneqq \pi_0\times R\left(\rho\right)|\cdot|^{-\frac{p_1-1}{2}}$. Substituting this back we get
\[
    \Hom_{G_n}\left(\rho|\cdot|^{\frac{p_1-1}{2}}\times\Pi,\sigma\right) \neq 0.
\]
Consider the restriction of the induced representation $\rho|\cdot|^{\frac{p_1-1}{2}}\times\Pi$ to the mirabolic subgroup $M_{n+1}$ using our discussions in Section \ref{sub:orbit_analysis_on_grassmannian_rank_1_case}. The contribution of the closed orbit (and derivatives) has a filtration, with each graded piece
    \[
        \left(\Sym^j\std_{k,\C}\otimes\rho\right)|{\det}_k|^{\frac{p_1}{2}}\times\left(\Pi~\big|_{G_{n-k}}\right)
    \]
as representations of $G_n$, where $j=0,1,2,\cdots$. Note that $\Sym^j\std_{k,\C}\otimes\rho$ is still a finite length representation of $G_k$, and its Jordan--H\"older factors are of the form $\tau|\det_k|^s$, where $\tau$ is an irreducible representation of $G_k$ with unitary central character, and $s\in\R_{\geq 0}$ is a non-negative real number. Applying Lemma \ref{L:vanishingLem1} we get
\[
    \Hom_{G_n}\left(\left(\Sym^j\std_{k,\C}\otimes\rho\right)|{\det}_k|^{\frac{p_1}{2}}\times\left(\Pi~\big|_{G_{n-k}}\right), \sigma\right) = 0
\]
for all $j=0,1,2,\cdots$. Hence, the closed orbit does not contribute to the Hom space. On the other hand, the contribution of the open orbit is
\[
    \leftindex^u\ind_{S_k}^{M_{n+1}}\left(\rho|{\det}_k|^{\frac{n+p_1-k}{2}}~\big|_{M_k}\boxtimes\Pi\,|{\det}_{n+1-k}|^{-\frac{k}{2}}\right).
\]
We now discuss case by case.

\vskip 5pt

$\bullet$ \textit{Case 1}: $k=1$, so $\rho$ is a character of $G_1$. In this case $M_k=1$, so the contribution of the open orbit specializes to
\[
    \leftindex^u\ind_{G_n}^{M_{n+1}}\Pi\,|{\det}_n|^{-\frac{1}{2}}.
\]
When further restricted to $G_n$, it becomes to $\Pi\otimes\omega_n$, so we get
\[
    \Hom_{G_n}\left(\Pi\otimes\omega_n, \sigma\right) \neq 0.
\]

\vskip 5pt

$\bullet$ \textit{Case 2}: $k=2$, and $\rho = \delta_1|\det|^s$ is an essential discrete series representation of $G_2$. Here $s=0$ if $\scrL_1=\scrL\left(\delta_1\right)$, and $s=s_1$ if $\scrL_1=\scrL\left(\delta_1,s_1\right)$. In this case the contribution of the open orbit specializes to
\[
    \leftindex^u\ind_{S_2}^{M_{n+1}}\left(\delta_1|{\det}_2|^{s+\frac{n+p_1-2}{2}}~\big|_{M_2}\boxtimes\Pi\,|{\det}_{n-1}|^{-1}\right).
\]
Now it is the time to apply Lemma \ref{L:Res.d.s.To.Mirabolic}. For each graded piece $\chi\,{\rm sgn}^{j}|\cdot|^{j+\frac{m-1}{2}}$ of $\Gamma_{\delta_1}$, we get a corresponding piece of the open orbit
\[
    \leftindex^u\ind_{S_2}^{M_{n+1}}\left(\chi\,{\rm sgn}^{j}|\cdot|^{s+j+\frac{n+p_1+m-3}{2}}\boxtimes\Pi\,|{\det}_{n-1}|^{-1}\right) 
    \simeq \chi\,{\rm sgn}^{j}|\cdot|^{j+\frac{p_1+m-2}{2}}\times\left(\Pi\otimes\omega_{n-1}\right).
\]
Here we regard both sides as $G_n$-representations. Once again by Lemma \ref{L:vanishingLem1}, we have
\[
    \Hom_{G_n}\left(\chi\,{\rm sgn}^{j}|\cdot|^{s+j+\frac{p_1+m-2}{2}}\times\left(\Pi\otimes\omega_{n-1}\right),\sigma\right) = 0.
\]
Thus it remains to compute the piece of open orbit corresponding to $\Gamma_{op}$. It is not hard to see that
\[
    \leftindex^u\ind_{S_2}^{M_{n+1}}\left(\Gamma_{op}|\cdot|^{s+\frac{n+p_1-2}{2}}\boxtimes\Pi\,|{\det}_{n-1}|^{-1}\right) 
    \simeq \left(\chi|\cdot|^{s+\frac{p_1+m-1}{2}}\times\Pi \right)\otimes\omega_n.
\]
Here again we regard both sides as $G_n$-representations. From the discussions above we have
\[
    \Hom_{G_n}\left(\left(\chi|\cdot|^{s+\frac{p_1+m-1}{2}}\times\Pi \right)\otimes\omega_n,\sigma\right).
\]
By Lemma \ref{L:FJToBessel}, we can pick up a unitary character $\xi_1=|\cdot|^{s_1}$ of $G_1$, where $s_1\in\sqrt{-1}\cdot\R$, such that $\xi_1$ ``commutes'' with $\chi|\cdot|^{s+\frac{p_1+m-1}{2}}$ in the sense that
\[
    \xi_1\times \chi|\cdot|^{s+\frac{p_1+m-1}{2}} \simeq \chi|\cdot|^{s+\frac{p_1+m-1}{2}} \times \xi_1,
\]
and there is an isomorphism
\[
    \Hom_{G_n}\left(\left(\chi|\cdot|^{s+\frac{p_1+m-1}{2}}\times\Pi \right)\otimes\omega_n,\sigma\right) \simeq \Hom_{G_n}\left(\xi_1\times\left(\chi|\cdot|^{s+\frac{p_1+m-1}{2}}\times\Pi \right),\sigma\right).
\]
Then exchange the induction order and play the same game again for 
\[
    \xi_1\times\left(\chi|\cdot|^{s+\frac{p_1+m-1}{2}}\times\Pi \right) \simeq \chi|\cdot|^{s+\frac{p_1+m-1}{2}}\times\left(\xi_1\times\Pi \right),
\]
it follows that
\[
    \Hom_{G_n}\left(\left(\xi_1\times\Pi\right)\otimes\omega_n, \sigma\right) \neq 0.
\]

\vskip 5pt

To conclude, in any case of above, eventually we find a unitary character $\xi$ of $G_{k-1}$, such that
\[
    \Hom_{G_n}\left(\left(\xi\times\Pi\right)\otimes\omega_n,\sigma\right) \neq 0.
\]
Next we deal with $\Pi = \pi_0\times R\left(\rho\right)|\cdot|^{-\frac{p_1-1}{2}}$. Apply the MVW-involution
\[
    c: G_n \lra G_n, \quad g \longmapsto w_n\cdot\leftindex^{t}g^{-1}\cdot w_n^{-1}
\]
to this Hom space, we get
\[
    \Hom_{G_n}\left(\left(R\left(\rho\right)^\vee|\cdot|^{\frac{p_1-1}{2}}\times\pi_0^\vee\times \xi^{-1}\right)\otimes\omega_n ,\sigma^\vee \right) \neq 0.
\]
Here we have made use of two facts:

\begin{itemize}
    \item the MVW-involution takes any irreducible representation $\tau$ to its contragredient $\tau^\vee$, and takes parabolic induction $\tau_1\times\tau_2$ to $\tau_2^c\times\tau_1^c$;

    \vskip 5pt

    \item the Weil representation $\omega_n$ is invariant under the MVW-involution, see \cite[Lem. 4.3]{MR3930015} for more details.
\end{itemize}
\vskip 5pt
Once again, with the help of Lemma \ref{L:FJToBessel}, one can see that we are in the same situation as at the beginning of this section. Repeating what we have done, eventually we obtain 
\begin{equation}\label{E:ReductionIPeriodnonzero}
    \Hom_{G_n}\left({\rm Ps}_{2k}\times\pi_0,\sigma\right) \neq 0
\end{equation}
for a suitably chosen unitary principal series representation ${\rm Ps}_{2k}$ of $G_{2k}$. Indeed, if we write 
\[
    {\rm Ps}_{2k} = \xi_1\times\xi_2\times\cdots\times\xi_{2k},
\]
we can require that each unitary character $\xi_\alpha$ appearing in the inducing data of ${\rm Ps}_{2k}$ does not correspond to any $\scrL_i$ or $\scrL'_j$.

\vskip 5pt

\subsection{Annihilator varieties of distinguished representations} 
\label{sub:annihilator_varieties_of_distinguished_representations}

We now recall a key ingredient of the proof: the \textit{annihilator variety}, and its connection with $\SL_2$-type. Let $\pi$ be a finite length representation of $G_n$ and $\pi^{\HC}$ the Harish-Chandra module of $\pi$. Let $\gothg_n$ be the complexified Lie algebra of $G_n$, and $\gothg_n^*$ the dual Lie algebra. The annihilator variety of $\pi$ is defined to be the zero locus in $\gothg_n^*$ of 
\[
    \calV\left(\pi\right) = {\rm{Zero}}\left({\rm gr}{\rm Ann}\left(\pi^{\HC}\right)\right).
\]
It turns out that $\calV\left(\pi\right)$ always lives in the nilpotent cone $\calN_n\subset\gothg_n^*$, and when $\pi$ is irreducible there is a coadjoint orbit $\calO_\lambda$ (which corresponds to a partition $\lambda$ of $n$), such that
\[
    \calV\left(\pi\right) = \overline{\calO_\lambda}.
\]
Following \cite[Def. 1.2.2]{MR3029948}, we call $\lambda$ the \textit{associated partition} of $\pi$, and denote it by $AP\left(\pi\right)$. Recall that when $\pi$ is unitary, we have attached another partition $TP\left(\pi\right)$ to it. The following result of Gourevitch--Sahi \cite[Thm. B]{MR3029948} illustrates the relation of these two partitions. 

\begin{Thm}
Let $\pi$ be an irreducible unitary representation of $G_n$. Then
\[
    TP(\pi) = AP(\pi)^t,
\]
where the superscript ``$t$'' means taking the transpose. 
\end{Thm}

\vskip 5pt

We make some conventions for partitions. Given a partition $\lambda$ of $n$, we shall always write it in descending order and denote by $\lambda_i$ the $i$-th coordinate of $\lambda$. If $i$ is greater than the length of $\lambda$, i.e. $i>\lambda^t_1$, $\lambda_i$ should be understood as $0$. In \cite{MR4291954}, Gourevitch--Sayag--Karshon studied annihilator varieties associated to distinguished representations and obtained a very elegant ``micro-local'' necessary condition \cite[Thm. E]{MR4291954}.

\begin{Thm}
Let $\pi$ be an irreducible representation of $G_{n+1}$, and $\lambda=AP(\pi)$. Likewise, let $\sigma$ be an irreducible representation of $G_m$, and $\mu=AP(\sigma)$. If 
\[
    \Hom_{RS_{m,n}}\left(\pi,\sigma\boxtimes\psi_{m,n}\right) \neq 0,
\]
then we have $|\lambda_i^t-\mu_i^t| \leq 1$ for any $i\geq 1$.
\end{Thm}

\vskip 5pt

Combining these two theorems, one gets:

\begin{Cor}\label{C:GSKTypeInequality}
Let $\pi$ be an irreducible unitary representation of $G_{n+1}$, and $\sigma$ an irreducible unitary representation of $G_m$. If 
\[
    \Hom_{RS_{m,n}}\left(\pi,\sigma\boxtimes\psi_{m,n}\right) \neq 0,
\]
then $TP(\pi)$ and $TP(\sigma)$ are ``close'', in the sense that $|TP(\pi)_i-TP(\sigma)_i| \leq 1$ for any $i\geq 1$.
\end{Cor}

\vskip 5pt

This notion of being ``close'' is already very close to the notion of ``relevance''!


\subsection{Information from annihilator varieties} 
\label{sub:information_from_annihilator_varieties}

Note that ${\rm Ps}_{2k}\times\pi_0$ is ``more tempered'' than $\pi$, in the sense that $NT_{{\rm Ps}_{2k}\times\pi_0} < NT_\pi$. So by induction hypothesis we conclude from (\ref{E:ReductionIPeriodnonzero}) that ${\rm Ps}_{2k}\times\pi_0$ and $\sigma$ are relevant. Now it is the time to use our combinatorial description of the notion of relevance. Note that
\[
    \Lambda\left(\eta,a;\pi,\sigma\right) \geq \Lambda\left(\eta,a;{\rm Ps}_{2k}\times\pi_0,\sigma\right)
\]
unless $\eta$ corresponds to one of $\left\{\xi_1,\xi_2,\cdots,\xi_{2k}\right\}$, so $\Lambda\left(\eta,a;\pi,\sigma\right) \geq 0$. On the other hand
\[
    \Lambda\left(\eta,a;\sigma,\pi\right) = \Lambda\left(\eta,a;\sigma,{\rm Ps}_{2k}\times\pi_0\right)
\]
as long as $\eta \neq \scrL_1$ or $a\neq p_1-1$. When $\eta = \scrL_1$ and $a = p_1-1$, we have
\[
    \Lambda(\scrL_1,p_1-1;\sigma,\pi) = 
    \begin{cases}
        m(\scrL_1,p_1-1;\sigma) - m(\scrL_1,p_1;\pi), & \textit{if }\scrL_1=\scrL(\delta_1);
        \vspace{0.5em}\\

        m(\scrL_1,p_1-1;\sigma) + m(D(\scrL_1),p_1;\sigma) - m(\scrL_1,p_1;\pi),  & \textit{if }\scrL_1=\scrL(\delta_1,s_1).
    \end{cases}
    .
\]
So by Lemma \ref{L:RelevantCombinatorial}, to show $\pi$ and $\sigma$ are relevant it only remains to show that 
\begin{equation*}\label{E:multi-ineq}\tag{$\heartsuit$}
    \Lambda\left(\scrL_1,p_1-1;\sigma,\pi\right) \geq 0.
\end{equation*}

\vskip 5pt

Now we make some further simplifications. 

\vskip 5pt

$\bullet$ \textit{Step 1}: Firstly, we can assume that for any $i\in\left\{2,\cdots,r\right\}$, if $p_i=p_1$ then $\scrL_i = \scrL_1$. This is because if for some $i$ we have $p_i=p_1$ but $\scrL_i \neq \scrL_1$, then we can apply the same argument as in Section \ref{sub:reduction_i_the_easy_step} to the block $\scrL_i\boxtimes S_{p_i}$. We will get another ``more tempered'' representation ${\rm Ps}'_{2k'}\times \pi'_0$, 
where $\pi'_0$ has parameter 
\[
    \begin{cases}
        \Psi_\pi-\scrL_i\boxtimes S_{p_i} + \scrL_i\boxtimes S_{p_i-2}, \quad & \textit{if }\scrL_i=\scrL\left(\delta_i\right);
        \vspace{0.5em}\\

        \Psi_\pi-\scrL_i\boxtimes S_{p_i}+D\left(\scrL_i\right)\boxtimes S_{p_i-1}, \quad & \textit{if }\scrL_i=\scrL\left(\delta_i,s_i\right),
    \end{cases}
\]
and we still have
\[
    \Hom_{G_n}\left({\rm Ps}'_{2k'}\times\pi'_0,\sigma\right) \neq 0.
\]
By the induction hypothesis, ${\rm Ps}'_{2k'}\times \pi'_0$ is relevant to $\sigma$. Then Lemma \ref{L:RelevantCombinatorial} implies that
\[
    \Lambda\left(\scrL_1,p_1-1;\sigma,{\rm Ps}'_{2k}\times \pi'_0\right) = \Lambda\left(\scrL_1,p_1-1;\sigma,\pi\right)
\]
is non-negative, i.e. the desired inequality (\ref{E:multi-ineq}) holds. 

\vskip 5pt

$\bullet$ \textit{Step 2}: Secondly, with the assumption above, we can further assume that for any $j\in\left\{1,2,\cdots,s\right\}$, we have $q_j<p_1$. The reason is the same as above. 

\vskip 5pt


The main points of making these simplifications are as follows. First of all, under these assumptions, we have
\[
    \Lambda\left(\scrL_1,p_1-1;\sigma,\pi\right) = m\left(\scrL_1,p_1-1;\sigma\right) - m\left(\scrL_1,p_1;\pi\right).
\]
Next, $p_1$ is the biggest number in the $\SL_2$-type of $\pi$, i.e. $TP\left(\pi\right)$ is of the form
\[
    \big(\underbrace{p_1,\cdots,p_1}_{m_\pi\textit{-times}},\cdots\big),
\]
and the number of occurrences of $p_1$ in $TP\left(\pi\right)$ is exactly $m_\pi=m\left(\scrL_1,p_1;\pi\right)\cdot d_1$, where $d_1 = k$ or $2k$, depending on the parameter $\scrL_1$ is of the form $\scrL\left(\delta_1\right)$ or $\scrL\left(\delta_1,s_1\right)$.
Similarly, $p_1-1$ is the biggest number in the $\SL_2$-type of $\sigma$, i.e. $TP\left(\sigma\right)$ is of the form
\[
    \big(\underbrace{p_1-1,\cdots,p_1-1}_{m_\sigma\textit{-times}},\cdots\big).
\]
We would like to express the number of occurrences $m_\sigma$ in terms of $m\left(\scrL_1,p_1-1;\sigma\right)$, as we did for $m_\pi$. There are two cases.

\vskip 5pt

$\bullet$ \textit{Case 1}: $p_1-1>1$. In this case, similar to discussions in Step 1, we can even further assume that: if $q_j=p_1-1$ then $\scrL'_j = \scrL_1$. Then the number of occurrences of $p_1-1$ in $TP\left(\sigma\right)$ is
\[
    m_\sigma=m\left(\chi_1,p_1-1;\sigma\right)\cdot d_1.
\]
By Corollary \ref{C:GSKTypeInequality}, we know that $TP\left(\pi\right)$ and $TP\left(\sigma\right)$ must be ``close'', in the sense that the differences at each coordinate should be less than or equal to 1. This implies that $m_\sigma\geq m_\pi$, hence we get
\[
    m\left(\scrL_1,p_1-1;\sigma\right) \geq m\left(\scrL_1,p_1;\pi\right)
\] 
as desired.

\vskip 5pt

$\bullet$ \textit{Case 2}: $p_1-1=1$. In this case, similar to discussions in Step 1, we can replace $\sigma$ in (\ref{E:ReductionIPeriodnonzero}) by irreducible unitary representations of the form
\[
    \sigma_0 \times I\left(s_1,s_2,\cdots,s_\ell\right),
\] 
where $\sigma_0$ is an irreducible unitary representation corresponding to the parameter
\[
    \Psi_{\sigma_0} = \underbrace{\scrL_1+\cdots+\scrL_1}_{m\left(\scrL_1,p_1-1;\sigma\right)\textit{-times}},
\]
and $s_1,s_2,\cdots,s_\ell$ can take arbitrary values in $\sqrt{-1}\cdot\R^\ell$ except for a measure-zero subset ${\rm Ex}\subset\sqrt{-1}\cdot\R^\ell$. Then, by Lemma \ref{L:FJToBessel} and Proposition \ref{P:ReductionToCorank1}, we deduce from (\ref{E:ReductionIPeriodnonzero}) that
\[
    \Hom_{RS_{m,n}}\left(\pi,\sigma_0\boxtimes\psi_{m,n}\right)\neq 0,
\]
where $m = m\left(\scrL_1,p_1-1;\sigma\right) \cdot d_1$. Note that the $\SL_2$-type $TP\left(\sigma_0\right)$ of $\sigma_0$ is simply
\[
    \big(\underbrace{p_1-1,\cdots,p_1-1}_{m\textit{-times}}\big).
\]
Again, by Corollary \ref{C:GSKTypeInequality}, we know that $TP\left(\pi\right)$ and $TP\left(\sigma_0\right)$ must be ``close,'' and it follows that 
\[
    m\left(\scrL_1,p_1-1;\sigma\right) \geq m\left(\scrL_1,p_1;\pi\right)
\] 
as desired.

\vskip 5pt

This completes the proof of Theorem \ref{T:MainNTGGPGL}(1).

\vskip 10pt






\vskip 15pt

\bibliographystyle{alpha}
\bibliography{NTGGPGLArchiRef}

\end{document}